\documentclass[11pt]{article}

\usepackage{amssymb,latexsym,amsmath}
\usepackage{epsfig}
\usepackage{caption}

\def\Pr{\mathop{\rm Pr}}

\def\B{{\mathcal B}}

\def\P{{\mathcal P}}

\def\sX{{\mathds X}}

\def\sU{{\mathds U}}

\def\sU{{\mathds U}}

\usepackage{amsmath}
\usepackage{amsfonts}
\usepackage{array}
\usepackage{dsfont}
\usepackage{hyperref}
\usepackage{amssymb}

\usepackage{amsthm}
\usepackage{bbold}
\usepackage{caption}
\usepackage{standalone}
\usepackage{accents}
\usepackage{enumitem}
\usepackage{tikz}
\usepackage{pgfplots}
\pgfplotsset{compat=1.6}
\usepgfplotslibrary{groupplots}
\usepackage{cancel}
\allowdisplaybreaks

\newtheorem{lemma}{Lemma}
\newtheorem{theorem}{Theorem}
\newtheorem{prop}{Proposition}
\newtheorem{assumption}{Assumption}
\newtheorem{corollary}{\bf{Corollary}}[section]

\theoremstyle{definition}
\newtheorem{example}{Example}
\newtheorem{definition}{Definition}

\theoremstyle{remark}
\newtheorem*{remark}{Remark}

\usepackage{amssymb,latexsym,amsmath}
\usepackage{mathrsfs}
\usepackage{epsfig}
\usepackage{caption}

\newcommand{\R}{\mathds{R}}
\newcommand{\Zplus}{\mathds{Z}_+}
\newcommand{\N}{\mathds{N}}

\newcommand{\dd}{\mathrm{d}}

\newcommand{\sy}[1]{{\color{black} #1}}
\newcommand{\adk}[1]{{\color{black} #1}}

\pgfplotsset{soldot/.style={color=blue,only marks,mark=*}}
\pgfplotsset{holdot/.style={color=blue,fill=white,only marks,mark=*}}
\fussy

\allowdisplaybreaks

\begin{document}
\sloppy
\title{Near Optimal Approximations and Finite Memory Policies for POMPDs with Continuous Spaces \\
{\small {Dedicated to Professor Peter E. Caines, on the occasion of his 80th birthday.}}

%\title{Convergence of Finite Memory Q-Learning for POMDPs and Near Optimality of Learned Policies under Filter Stability
\thanks{E. Bayraktar is partially supported by the National Science Foundation under grant DMS-2106556 and by the Susan M. Smith chair. \\
S. Yuksel is partially supported by
the Natural Sciences and Engineering Research Council (NSERC) of Canada.}
}

\author{Ali Devran Kara, Erhan Bayraktar and  Serdar Y\"uksel 
\thanks{A. D. Kara is with the Department of Mathematics, Florida State University, Tallahassee, USA, Email: \{akara@fsu.edu\}, E. Bayraktar is with the Department of Mathematics, University of Michigan, Ann Arbor, USA, Email: \{erhan@umich.edu\}. S. Y\"uksel is with the Department of Mathematics and Statistics,
     Queen's University, Kingston, ON, Canada,
     Email: \{yuksel@queensu.ca\}}
     }
\maketitle

\begin{abstract}

We study an approximation method for partially observed Markov decision processes (POMDPs) with continuous spaces. Belief MDP reduction, which has been the standard approach to study POMDPs requires rigorous approximation methods for practical applications, due to the state space being lifted to the space of probability measures. Generalizing recent work, in this paper we present rigorous approximation methods via discretizing the observation space and constructing a fully observed finite MDP model using a finite length history of the discrete observations and control actions. We show that the resulting policy is near-optimal under some regularity assumptions on the channel, and under certain controlled filter stability requirements for the hidden state process. Furthermore, by quantizing the measurements, we are able to utilize refined filter stability conditions. We also provide a Q learning algorithm that uses a finite memory of discretized information variables, and prove its convergence to the optimality equation of the finite fully observed MDP constructed using the approximation method.
\end{abstract}

%*************************************************************************************************************
% \biaoti{THE CAPITALIZED TITLE OF YOUR ARTICLE$^*$}{The list of authors' names with the LAST NAME capitalized
% and the authors' names should be separated by "\cdd"}{the first author's name \\ the first author's affiliation
% and Email address\\ the second author's name\\ the second author's affiliation. More can be listed like this.}
% {$^*$ The titles and numbers of the foundations that support this article.}
%*************************************************************************************************************

\section{Introduction}

Partially Observed Markov Decision Problems (POMDPs) offer a practically rich and relevant, and mathematically challenging model. Analysis of optimality and computation of solutions become complex even for finite models.
Hence, approximations are required for efficient (near) optimality and control analysis. We present a brief review in the following.

Most of the studies in the literature are algorithmic and computational contributions such as \cite{porta2006point,ZhHa01}. These studies develop computational algorithms, utilizing structural convexity/concavity properties of the value function under the discounted cost criterion. \cite{spaan2005perseus} provides an insightful algorithm which may be regarded as a quantization of the belief space; however, no rigorous convergence results are provided. \cite{smith2012point,pineau2006anytime} also present quantization based algorithms for the belief state, where the state, measurement, and the action sets are finite. 

%\adk{Another line of research for approximation techniques in POMDPs is point based algorithms (see the survey paper (\cite{shani2013survey})). For point based algorithms, approximation is also done via choosing a finite subset if the belief state space and the value functions are calculated for this finite subset. The focus of the algorithms is to choose the finite belief subset to make the algorithms efficient in terms of computation and approximation accuracy.}
Rigorous approximation techniques for POMDPs have mostly focused on finite state or finite observation models. Some works that study POMDPs with continuous state or observation spaces include \cite{SYLTAC2017POMDP,SaYuLi15c,zhou2008density,zhou2010solving}. 

\cite{SYLTAC2017POMDP,SaYuLi15c} introduce a rigorous approximation analysis after establishing weak continuity conditions on the transition kernel defining the (belief-MDP) via the non-linear filter (\cite{FeKaZa12, KSYWeakFellerSysCont}), and  it is shown that the finite model approximations obtained through quantization are asymptotically optimal as the number of quantization bins increases. The finite model is constructed by choosing a finite number of belief points which are sufficiently close to any point in the original belief space, however, this may not be an easy task in general, as the belief space consists of probability measures.

Another rigorous set of studies with continuous observation and state spaces is \cite{zhou2008density,zhou2010solving} where the authors provide an explicit quantization method for the set of probability measures containing the belief states, where the state space is parametrically representable under strong density regularity conditions. The quantization is done through the approximations as measured by Kullback-Leibler divergence (relative entropy) between probability density functions. %The near optimality of the approximation is shown under the assumption that the parametric density family covers the original belief space with sufficiently small error and that the transition model of the parametric model is close enough to the transition model of the original POMDP. We note again that these assumptions, in general, are not straightforward to test. In particular, for the transition model assumption, one needs regularity assumptions for the belief transition kernel, which can only shown to be weakly continuous in general (see e.g. \cite{FeKaZa12, KSYWeakFellerSysCont}). 

 %Further recent studies include \cite{mao2020,Mahajan2019}. \cite{Mahajan2019} present a notion of approximate information variable and studies near optimality of policies that satisfies the approximate information state property. \cite{mao2020} analyzes a similar problem under a decentralized setup. Our explicit approximation results in this paper will find applications in both of these studies. 

We refer the reader to the survey papers by \cite{Lov91-(b),Whi91,hansen2013solving} and \cite{Kri16} for further structural results as well as algorithmic and computational methods for approximating POMDPs. Notably, for POMDPs \cite{Kri16} presents structural results on optimal policies under monotonicity conditions of the value function in the belief variable.

\cite{yu2008near} studies near optimality of finite window policies for average cost problems where the state, action and observation spaces are finite; under the condition that the liminf and limsup of the average cost are equal and independent of the initial state, the paper establishes the near-optimality of (non-stationary) finite memory policies.
In another related direction, \cite{white1994finite} study finite memory approximation techniques for POMDPs with finite state, action and measurements. The POMDP is reduced to a belief MDP and the worst and best case predictors prior to the $N$ most recent information variables are considered to build an approximate belief MDP. The original value function is bounded using these approximate belief MDPs that use only finite memory, where the finiteness of the state space is critically used. 

Our paper generalizes several recent studies \cite{kara2021convergence,kara2020near,demirci2023average,demirciRefined2023} and the review paper \cite{tutorialkara2024partially}. In these studies, finite memory approximations as well as quantized belief approximations are studied, and near optimality of the approximations and their associated Q learning algorithms are provided under certain controlled filter stability assumptions, both for discounted and average cost criteria. However, the methods provided in these papers are tailored towards finite observation spaces, and the efficiency of the algorithms can be affected significantly for large observation spaces. The current paper addresses this gap. 

{We dedicate this paper to Professor Peter E. Caines, who has shaped stochastic control theory fundamentally in a variety of directions. In the context of this paper on control with partial information, the contributions of Professor Caines in a variety of contexts include: filtering theory and associated stability analysis \cite{Caines}, partially observable jump parameter systems \cite{caines1995adaptive}, control with Poisson measurements \cite{ades2000stochastic}, mean-field games with noisy measurements \cite{sen2019mean}, observer design for automata \cite{caines1991classical},  and filtering in Riemannian manifolds \cite{ng1984nonlinear}. A particular connection with our current paper is that observability entails filter stability (\cite{van2010nonlinear,mcdonald2018stability,MYRobustControlledFS}) which then leads to near optimality of sliding window control policies, not unlike the case in classical stochastic linear systems theory \cite{Caines} where observability implies near optimality of finite memory output feedback. Our paper makes this analysis precise for a large class of systems.}

\sy{
{\bf Contributions.} In this paper we develop rigorous approximations and learning results for POMDPs with Borel state, action, and measurement spaces. Furthermore, under regularity conditions, we obtain rigorous (finite dimensional) approximations, and under explicit filter stability conditions, we obtain near optimality bounds of finite window based policies which then leads to finite model approximations. 

Notably, we quantize either the state space or the measurement space, \sy{noting that it is well known by now that action spaces can be quantized with arbitrarily low loss in performance under weak continuity conditions \cite{saldi2014near}\cite[Theorem 3.16]{SaLiYuSpringer}}

Quantization of the hidden state space results in a belief state process which takes values from a simplex. Hence, the problem can be solved via this simplex valued approximate belief state and leads to a finite-dimensional filter which can be computer with relative ease. We present sufficient conditions to guarantee the near optimality of this discretization method in Section \ref{hidden_state_disc}.

Furthermore, in Section \ref{FiniteMemoryCompSec}, we study the performance of finite window based policies and relate the performance loss due to finite window restriction to controlled filter stability. 

In Section \ref{obs_disc}, towards a finite model approximation, we quantize the measurements and obtain performance loss bounds due to the approximations of the measurements to finite ones. This quantization leads to a garbling of information and therefore the performance loss can be obtained with such a perspective, as we show in the paper.

The analysis above culminates in a joint analysis of the performance loss due to measurement quantization and finite window restriction in Section \ref{FiniteMemDiscreteObs}. 

By quantizing the measurement space we are able to utilize refined filter stability conditions. In particular, discretization on the observation variables results in an approximate non-linear filter where the conditioning is based on a less informative information process. We present explicit conditions under which the approximate filter is stable as well. This in turn implies near-optimality of the discretized finite-window information variables.

Finally, Section \ref{QLearningApproximateM} applies the analysis above to reinforcement learning via finite memory Q-learning, which is shown to converge to near optimality under the conditions noted above.

Compared to the results in the literature, this paper is the first one to our knowledge which considers the case with continuous space measurements and establishes explicit conditions on the near optimality of finite window policies, and also with a learning algorithm under convergence guarantees to near optimality.

} %We focus on an approximation technique for POMDPs with continuous spaces. The approximate model uses a finite subset of the original belief space; we provide a constructive method for the approximate model by first discretizing the observation variables and using a finite history of the discretized information variables. Different from the past works with continuous spaces, the method presented here can be shown to be nearly optimal under more general and easily testable conditions. In particular, the resulting model and the control policy can be shown to be nearly optimal if the measurement channel has a certain regularity property, and if the controlled filter is stable under discrete measurements. Our approach make use of the fact that, once the observations are discretized, one has another POMDP model with a different and degraded channel. Furthermore, we provide a Q learning algorithm for the approximate model and prove its convergence to near-optimality. 

%The provided method also leads to the following natural but interesting question: under computation or memory constraints, is it better to use a finer discretization scheme or use a longer memory for the discretized information variables for a better performance?

{
Before we start out analysis, building on \cite{KSYContQLearning}, we first review some technical tools we will need along the paper.

\subsection{Convergence Notions for Probability Measures and Regularity Properties of Transition Kernels}
For the analysis of the technical results, we will use different notions of convergence for sequences of probability measures.

Two important notions of convergence for sequences of probability measures are weak convergence and convergence under total variation. For some $N\in\N$, a sequence $\{\mu_n,n\in\N\}$ in $\mathcal{P}(\mathds{X})$ is said to converge to $\mu\in\mathcal{P}(\mathds{X})$ \emph{weakly} if $\int_{\mathds{X}}c(x)\mu_n(dx) \to \int_{\mathds{X}}c(x)\mu(dx)$ for every continuous and bounded $c:\mathds{X} \to \R$.
%One important property of weak convergence is that the space of probability measures on a complete, separable, and metric (Polish) space endowed with the topology of weak convergence is itself complete, separable, and metric \citep{Par67}. One such metric is the bounded Lipschitz metric  \cite[p.109]{villani2008optimal}, which is defined for $\mu,\nu \in \P(\mathds{X})$ as 
%\begin{equation}\label{BLmetric}
%\rho_{BL}(\mu,\nu):=\sup_{\|f\|_{BL}\leq1} | \int f d\mu - \int f d\nu | 
%\end{equation}
%where \[ \|f\|_{BL}:=\|f\|_\infty+\sup_{x\neq y}\frac{|f(x)-f(y)|}{d(x,y)} \]
%and $\|f\|_\infty=\sup_{x\in\mathds{X}}|f(x)|$.

  For probability measures $\mu,\nu \in \mathcal{P}(\mathds{X})$, the \emph{total variation} metric is given by
  \begin{align*}
    \|\mu-\nu\|_{TV}&=2\sup_{B\in\mathcal{B}(\mathds{X})}|\mu(B)-\nu(B)|=\sup_{f:\|f\|_\infty \leq 1}\left|\int f(x)\mu(\dd x)-\int f(x)\nu(\dd x)\right|,
  \end{align*}
  \noindent where the supremum is taken over all measurable real $f$ such that $\|f\|_\infty=\sup_{x\in\mathds{X}}|f(x)|\leq 1$. A sequence $\mu_n$ is said to converge in total variation to $\mu \in \mathcal{P}(\mathds{X})$ if $\|\mu_n-\mu\|_{TV}\to 0$.

 {Finally, for probability measures $\mu,\nu \in \mathcal{P}(\mathds{X})$ with finite first order moments (that is, $\int \|x\| \, d\nu$ and $\int \|x\| \, d\mu $ are finite)}, the \emph{first order Wasserstein} distance is defined as 
\begin{align*}
W_1(\mu,\nu)=\inf_{\Gamma(\mu,\nu)}E[|X-Y|]=\sup_{f: Lip(f)\leq 1}|\int f(x)\mu(dx)-\int f(x)\nu(dx)|
\end{align*}
where $\Gamma(\mu,\nu)$ denotes the all possible couplings of $X$ and $Y$ with marginals $X\sim\mu$ and $Y\sim\nu$, {and 
\begin{align}
Lip(f) := \sup_{e \neq e'} \frac{f(e) - f(e')}{\|e - e'\|},\nonumber
\end{align}}
 and the second {equality follows from the dual formulation of the Wasserstein distance \cite[Remark 6.5]{Vil09}}. Note that the weak convergence and the Wasserstein convergence are equivalent if the underlying space is compact.

We can now define the following regularity properties for the transition kernels:
\begin{itemize}
\item $\mathcal{T}(\cdot|x,u)$ is said to be weakly continuous in $(x,u)$, if $\mathcal{T}(\cdot|x_n,u_n)\to \mathcal{T}(\cdot|x,u)$ weakly for any $(x_n,u_n)\to (x,u)$.
\item $\mathcal{T}(\cdot|x,u)$ is said to be continuous under total variation in  $(x,u)$, if $\|\mathcal{T}(\cdot|x_n,u_n)- \mathcal{T}(\cdot|x,u)\|_{TV}\to 0$  for any $(x_n,u_n)\to (x,u)$.
\item $\mathcal{T}(\cdot|x,u)$ is said to be continuous under the first order Wasserstein distance in $(x,u)$, if \[W_1(\mathcal{T}(\cdot|x_n,u_n), \mathcal{T}(\cdot|x,u))\to 0\]  for any $(x_n,u_n)\to (x,u)$. {To ensure continuity of $\mathcal{T}$ with respect to the first order Wasserstein distance, in addition to weak continuity, we may assume that there exists a function $g:[0,\infty) \rightarrow [0,\infty)$ such that as $t \to \infty$, $\frac{g(t)}{t} \uparrow \infty$, and 
$$\sup_{(x,u) \in K \times \mathds{U}} \int g(\|y\|) \, \mathcal{T}(dy|x,u) < \infty$$ for any compact $K \subset \mathds{X}$. Note that the latter condition implies uniform integrability of the collection of random variables with probability measures ${\cal T}(dx_1|X_0=x_n,U_0=u_n)$ as $(x_n, u_n) \to (x,u)$, which coupled with weak convergence can be shown to imply convergence under the Wasserstein distance.}
\end{itemize}
\begin{example}\label{examples}
Some example models satisfying these regularity properties are as follows:
\begin{itemize}
\item[(i)] For a model with the dynamics $x_{t+1}=f(x_t,u_t,w_t)$, the induced transition kernel $\mathcal{T}(\cdot|x,u)$ is weakly continuous in $(x,u)$ if $f(x,u,w)$ is a continuous function of $(x,u)$, since for any continuous and bounded function $g$
\begin{align*}
&\int g(x_1)\mathcal{T}(dx_1|x_n,u_n)=\int g(f(x_n,u_n,w))\mu(dw)\\
&\to\int g(f(x,u,w))\mu(dw)=\int g(x_1)\mathcal{T}(dx_1|x,u)
\end{align*}
where $\mu$ denotes the probability measure of the noise process.
{If we also have that $\mathds{X}$ is compact, the transition kernel $\mathcal{T}(\cdot|x,u)$ is also continuous under the first order Wasserstein distance}.

\item[(ii)] For a model with the dynamics $x_{t+1}=f(x_t,u_t)+w_t$, the induced transition kernel $\mathcal{T}(\cdot|x,u)$ is continuous under total variation in $(x,u)$ if $f(x,u)$ is a continuous function of $(x,u)$, and $w_t$ admits a continuous density function. 
\item[(iii)] In general, if the transition kernel admits a continuous density function $f$ so that $\mathcal{T}(dx_1|x,u)=f(x_1,x,u)\lambda(dx_1)$, then $\mathcal{T}(dx_1|x,u)$ is continuous in total variation. This follows from an application of  Scheff\'e's Lemma \cite[Theorem 16.12]{Bil95}. In particular, we can write that
\begin{align*}
\|\mathcal{T}(\cdot|x_n,u_n)-\mathcal{T}(\cdot|x,u)\|_{TV}=\int_{\mathds{X}}|f(x_1,x_n,u_n)-f(x_1,x,u)|\lambda(dx_1)\to 0.
\end{align*}
\item[(iv)] For a model with the dynamics $x_{t+1}=f(x_t,u_t,w_t)$, if $f$ is Lipschitz continuous in $(x,u)$ pair such that, there exists some $\alpha<\infty$ with
\begin{align*}
|f(x_n,u_n,w)-f(x,u,w)|\leq \alpha\left(|x_n-x|+|u_n-u|\right),
\end{align*}
we can then bound the first order Wasserstein distance between the corresponding kernels with $\alpha$:
\begin{align*}
&W_1\left(\mathcal{T}(\cdot|x_n,u_n),\mathcal{T}(\cdot|x,u)\right)=\sup_{Lip(g)\leq 1}\left|\int g(x_1)\mathcal{T}(dx_1|x_n,u_n)-\int g(x_1)\mathcal{T}(dx_1|x,u) \right|\\
&=\sup_{Lip(g)\leq 1}\left|\int g(f(x_n,u_n,w))\mu(dw)-\int g(f(x,u,w))\mu(dw)\right|\\
&\leq \int \left|f(x_n,u_n,w)-f(x,u,w)\right| \mu(dw)\leq \alpha\left(|x_n-x|+|u_n-u|\right).
\end{align*}
\end{itemize}
\end{example}
}

\section{Partially Observed Markov Decision Processes and Standard Solution Methods}
Let $\mathds{X} \subset \mathds{R}^m$ denote a Borel set which is the state space of a partially observed controlled Markov process for some $m\in\mathds{N}$. Let
$\mathds{Y} \subset \mathds{R}^n$ be another Borel set denoting the observation space of the model, for some $n\in\mathds{N}$, and let the state be observed through an
observation channel $O$. The observation channel, $O$, is defined as a stochastic kernel (regular
conditional probability) from  $\mathds{X}$ to $\mathds{Y}$, such that
$O(\,\cdot\,|x)$ is a probability measure on the sigma algebra $\mathcal{B}(\mathds{Y})$ of $\mathds{Y}$ for every $x
\in \mathds{X}$, and $O(A|\,\cdot\,): \mathds{X}\to [0,1]$ is a Borel
measurable function for every $A \in \mathcal{B}(\mathds{Y})$. 

We will assume that the channel admits a density function $g(x,y)$ with respect to a reference measure $\lambda(dy)\in\P(\mathds{Y})$ such that
\begin{align*}
O(dy|x)=g(x,y)\lambda(dy).
\end{align*}
 %A
%decision maker (DM) is located at the output of the channel $O$, and hence it only sees the observations $\{Y_t,\, t\in \Zplus\}$ and chooses its actions from 
$\mathds{U}\subset \mathds{R}^k$ denotes the action space for some $k\in\mathds{N}$.  

An {\em admissible policy} $\gamma$ is a
sequence of control functions $\{\gamma_t,\, t\in \Zplus\}$ such
that $\gamma_t$ is measurable with respect to the $\sigma$-algebra
generated by the information variables
$$
H_t=\{Y_{[0,t]},U_{[0,t-1]}\}, \quad t \in \mathds{N}, \quad
  \quad I_0=\{Y_0\},
$$
where $U_t=\gamma_t(I_t),\quad t\in \Zplus$,
are the $\mathds{U}$-valued control
actions and 
$Y_{[0,t]} = \{Y_s,\, 0 \leq s \leq t \}, \quad U_{[0,t-1]} =
  \{U_s, \, 0 \leq s \leq t-1 \}.$ When we consider a particular realization of the information variables, we use the notation $h_t$. We define $\Gamma$ to be the set of all such admissible policies. The update rules of the system are determined by the following relationships:
\[  \Pr\bigl( (X_0,Y_0)\in B \bigr) =  \int_B \mu(dx_0)O(dy_0|x_0), \quad B\in \mathcal{B}(\mathds{X}\times\mathds{Y}), \]
where $\mu$ is the (prior) distribution of the initial state $X_0$, and
\begin{eqnarray*}
\label{eq_evol}
 &\Pr\biggl( (X_t,Y_t)\in B \, \bigg|\, (X,Y,U)_{[0,t-1]}=(x,y,u)_{[0,t-1]} \biggr)\\
& = \int_B \mathcal{T}(dx_t|x_{t-1}, u_{t-1})O(dy_t|x_t),  
\end{eqnarray*}
$B\in \mathcal{B}(\mathds{X}\times\mathds{Y}), t\in \mathds{N},$ where $\mathcal{T}$ is the transition kernel of the model which is a stochastic kernel from $\mathds{X}\times
\mathds{U}$ to $\mathds{X}$.  We let the objective of the agent (decision maker) be the minimization of the infinite horizon discounted cost, 
  \begin{align}\label{criterion1}
    J_{\beta}(\mu,O,\gamma)= \sum_{t=0}^{\infty} \beta^tE_\mu^{{O},\gamma}\left[c(X_t,U_t)\right]
  \end{align}
 \noindent for some discount factor $\beta \in (0,1)$, over the set of admissible policies $\gamma\in\Gamma$, where $c:\mathds{X}\times\mathds{U}\to\R$ is a Borel-measurable stage-wise cost function.  $E_\mu^{{O},\gamma}$ denotes the expectation with initial state probability measure $\mu$, under the transition kernel $\mathcal{T}$ the channel $O$ and the policy $\gamma$. Note that $\mu\in\mathcal{P}(\mathds{X})$, where we let $\mathcal{P}(\mathds{X})$ denote the set of probability measures on $\mathds{X}$. We define the optimal cost for the discounted infinite horizon setup as a function of the priors and the measurement channel as
\begin{align*}
  J_{\beta}^*(\mu,{O})&=\inf_{\gamma\in\Gamma} J_{\beta}(\mu,{O},\gamma).
% \text{and }J_{\beta}^*(P,\cal{T})&=\inf_{\gamma\in\Gamma} J_{\beta}(P,\cal{T},\gamma)
\end{align*}

The solution of the problem in its current formulation requires one to find a control over the information variables 
\begin{align*}
h_t = \{y_0,\dots,y_{t}, u_0,\dots,u_{t-1}\}.
\end{align*}
Hence, the length of the information grown over time. In particular, the space $I_t$ lives in, $\mathds{Y}^{t}\times\mathds{U}^{t-1}$ grows exponentially. This leads one to use compressed versions of the original information variable $I_t$. A generic solution method for POMDPs then involves a compression scheme
\begin{align}\label{comp}
h_t \to z_t
\end{align} 
where the new state space $\mathcal{Z}$ does not grow over time and $z_t$ defines controlled Markov chain lets one to use standard methods to solve the optimal control problem. Ideally, the compression map (\ref{comp}) is  without loss of optimality. However, one might also work with lossy compression maps with controllable loss bounds, where the compressed state $z_t$ computationally appealing, which will be the main focus of this paper. 

In what follows, we will introduce different compression schemes, and present approximations with provable error bounds.

We start by introducing the most commonly used approach for the compression of the original information state, where the decision maker keeps track of the posterior distribution of the state $X_t$ given the available history $h_t$. In the following section, we formalize this approach.

\subsection{Reduction to Fully Observed Models Using Belief States}
It is by now a standard result that, for optimality analysis, every POMDP can be reduced to a completely observable Markov decision process (\cite{Yus76,Rhe74}), whose states are the posterior state distributions or {\it beliefs} of the observer or the filter process; that is, the state at time $t$ is
\begin{align}\label{belief_state}
z_t:=Pr\{X_{t} \in \,\cdot\, | Y_0,\ldots,Y_t, U_0, \ldots, U_{t-1}\} \in \P(\sX). 
\end{align}
We call this equivalent process the filter process \index{Belief-MDP}. The filter process lives in $\mathcal{Z} := \P(\sX)$ such that $\{z_t\}_t\subset \P(\mathds{X})$. Therefore, the state space of the new model can be viewed as $\mathcal{Z} = \P(\sX)$  and the action space remains as $\sU$. Here, $\mathcal{Z}$ is equipped with the Borel $\sigma$-algebra generated by the topology of weak convergence \cite{Bil99}. %Under this topology, $\mathcal{Z}$ is a standard Borel space \cite{Par67}. 
%The transition probability of the filter process can be constructed as follows: the joint conditional probability on next state and observation variables given the current control action and the current state of the filter process is given by
%\begin{align}\label{r_kernel1}
%R(B\times C|u_0,z_0) = \int_{\mathds{X}} \int_B Q(C|x_1,u_0)\mathcal{T}(dx_1|x_0,u_0)z_0(dx_0),
%\end{align}
%for all $B \in \B(\mathds{X})$ and $C \in \B(\mathds{Y})$. Then, the conditional distribution of the next observation variable given the current state of the filter process and the current control action is given by
%\begin{align*}
%P(C|u_0,z_0) = \int_{\mathds{X}} \int_{\mathds{X}} Q(C|x_1,u_0)\mathcal{T}(dx_1|x_0,u_0)z_0(dx_0),
%\end{align*}
%for all $C\in \B(\mathds{Y})$. Using this, we can disintegrate $R$ (see \cite[Proposition 7.27]{bertsekas78}) as follows:
%\begin{align}\label{r_kernel2}
%&R(B\times C|u_0,z_0) = \int_C F(B|y_1,u_0,z_0) P(dy_1|u_0,z_0) \nonumber \\
%&\phantom{xxxxxxxxxxxxxxxxxxxxxxxxx}=\int_C z_1(y_1,u_0,z_0)(B) P(dy_1|u_0,z_0),
%\end{align}
%where $F$ is a stochastic kernel from ${\cal P}(\mathds{X})\times \mathds{Y}\times\mathds{U}$ to $\mathds{X}$ and the posterior distribution of $x_1$, determined by the kernel $F$, is the state variable $z_1$ of the filter process.
%Using (\ref{r_kernel1}) and (\ref{r_kernel2}) and by defining $\mathcal{T}(dx_1|z_0,u_0):=\int_{\mathds{X}}\mathcal{T}(dx_1|x_0,u_0)z_0(dx_0)$, we can write
%\begin{align}\label{set_eq}
%\int_B Q(C|x_1)\mathcal{T}(dx_1|z_0,u_0)=\int_C z_1(y_1,u_0,z_0)(B) P(dy_1|u_0,z_0).
%\end{align}
Then, the transition probability $\eta$ of the filter process can be constructed as follows. If we define the measurable function 
\[F(z,u,y) := Pr\{X_{t+1} \in \,\cdot\, | Z_t = z, U_t = u, Y_{t+1} = y\}\]
 from ${\cal P}(\mathds{X})\times\mathds{U}\times\mathds{Y}$ to ${\cal P}(\mathds{X})$ and use the stochastic kernel $P(\,\cdot\, | z,u) = \Pr\{Y_{t+1} \in \,\cdot\, | Z_t = z, U_t = u\}$ from ${\cal P}(\mathds{X})\times\mathds{U}$ to $\mathds{Y}$, we can write $\eta$ as
\begin{align}
\eta(\,\cdot\,|z,u) = \int_{\mathds{Y}} 1_{\{F(z,u,y) \in \,\cdot\,\}} P(dy|z,u). \label{beliefK}
\end{align}
%that is, $\eta(\,\cdot\,|z,u)$ is an uncountable mixture of the probability measures $\{F(\,\cdot\,|y,u,z)\}_{y\in \mathds{Y}}$.
The one-stage cost function $\tilde{c}:{\cal P}(\mathds{X}) \times \mathds{U}\rightarrow[0,\infty)$ of the filter process is given by 
\begin{align}\label{belief_cost}
\tilde{c}(z,u) := \int_{\sX} c(x,u) z(dx),
\end{align}
which is a Borel measurable function. Hence, the filter process is a completely observable Markov process with the components $(\mathcal{Z},\sU,\tilde{c},\eta)$.

%For the filter process, the information variables is defined as
%\[
%\tilde{I}_t=\{Z_{[0,t]},U_{[0,t-1]}\}, \quad t \in \mathds{N}, \quad
  %\quad \tilde{I}_0=\{Z_0\}.
%\]

It is well known that an optimal control policy of the original POMDP can use the belief $Z_t$ as a sufficient statistic for optimal policies (see \cite{Yus76}, \cite{Rhe74}), provided they exist. More precisely, the filter process is equivalent to the original POMDP  in the sense that for any optimal policy for the filter process, one can construct a policy for the original POMDP which is optimal. Existence then follows under measurable selection conditions, e.g. satisfied by the aforementioned weak Feller continuity of the belief MDP (see \cite{HernandezLermaMCP}). In particular, this applies when the belief-MDP is weak Feller (\cite{FeKaZa12, KSYWeakFellerSysCont}), the action spaces are compact and the cost function is continuous and bounded.

Even though, the belief MDP approach provides a strong tool for the analysis of POMDPs,  solution of POMDPs may not be computationally feasible in general. This stems from mainly two reasons:
\begin{itemize}
\item Computation of the belief state $h_t\to z_t=P(X_t\cdot|h_t)$ is not tractable in general, except for a few special cases, e.g. Kalman filter for linear and Gaussian additive noise dynamics, or for finite models.
\item  The belief space $\mathcal{Z}=\P(\mathds{X})$ is always uncountable even when $\mathds{X}$, $\mathds{Y}$
 and  $\mathds{U}$ are finite. Thus, even one is able to track and compute the belief state, the solution of the control may not be practical. 
\end{itemize}

% it is usually too complicated computationally. The belief space $\mathcal{Z}=\P(\mathds{X})$ is always uncountable even when $\mathds{X}$, $\mathds{Y}$
 %and  $\mathds{U}$ are finite. Furthermore, the information variables $I_t$ grows with time and the computation of the belief state $Pr(X_t \in \cdot | I_t)$ can become intractable. 
Therefore, approximation of the belief-MDP is usually needed.

\subsection{Discretization of Action Sets}

For a weak Feller belief MDP (\cite{FeKaZa12, KSYWeakFellerSysCont}), \cite[Theorem 3.16]{SaLiYuSpringer} (see also \cite{SYL2016near}) has established near optimality of finite action policies. If $\mathds{U}$ is compact, a finite collection of action sets can be constructed, with arbitrary approximation error. Accordingly, we will assume that the action spaces are finite in the following.

\section{Approximation via Quantization of the (Hidden) State Space}\label{hidden_state_disc} The first approach we present builds on the discretization of the hidden state space $\mathds{X}$. This methods will rely on the model knowledge.  

We first focus on the discretization of the observation variables. For the discretization, we choose disjoint subsets $\{A_i\}_{i=0}^M\subset\mathds{X}$ such that $\cup_iA_i=\mathds{X}$. We define a finite set 
\begin{align*}
\mathds{X}_M:=\{x_0,\dots,x_M\}.
\end{align*}
Furthermore, the discretization map $\psi_{\mathds{X}}:\mathds{X}\to \hat{\mathds{X}}_M$ is defined such that 
\begin{align}\label{quant_map}
\psi_{\mathds{X}}(x)=x_i, \text{ if } x\in A_i.
\end{align}
In other words, the discretization checks what bin or set $x$ belongs to and maps it to the representative element of that bin. 

Using this discretization scheme, we introduce appropriately normalized transition and channel kernels $\hat{\mathcal{T}}$ and $\hat{O}$, following and generalizing  \cite{SaYuLi15c}\cite[Chapter 4]{SaLiYuSpringer}: We first choose a weighting measure $\pi(\cdot) \in \P(\mathds{X})$. For some $x\in A_i$, we define
\begin{align}\label{disc_hidden_model}
&\hat{\mathcal{T}}(\cdot|x,u) = \int_{A_i} \mathcal{T}(\cdot|z,u)\pi_i(dz)\nonumber\\
&\hat{O}(\cdot|x) = \int_{A_i} O(\cdot|z)\pi_i(dz)
\end{align}
where $\pi_i(\cdot)=\frac{\pi(\cdot)}{\pi(A_i)}$. Note that the approximate channel density can be written similarly as
\begin{align*}
\hat{g}(x,y) = \int_{A_i} g(z,y)\pi_z(dz).
\end{align*}

We also define the following auxiliary cost function:
\begin{align*}
\hat{c}(x,u) =  \int c(z,u)\pi_i(dz).
\end{align*}
These approximate kernel and functions can be seen to defined either on the finite set $\hat{\mathds{X}}$ or on the original space $\mathds{X}$ by extending them as constant over the quantization bins.

We define the infinite horizon discounted cost for this approximation by
  \begin{align}\label{criterion1}
    \hat{J}_{\beta}(\mu,\hat{O},\gamma)= \sum_{t=0}^{\infty} \beta^t\hat{E}_\mu^{\hat{O},\gamma}\left[\hat{c}(X_t,U_t)\right]
  \end{align}
 \noindent for some discount factor $\beta \in (0,1)$, over the set of admissible policies $\gamma\in\Gamma$, where  $\hat{E}_\mu^{\hat{O},\gamma}$ denotes the expectation with initial state probability measure $\mu$ and the transition kernel for the discretized model $\hat{\mathcal{T}}$ and the channel $\hat{O}$ under policy $\gamma$.  We define the optimal cost for the discretized model by
\begin{align*}
  \hat{J}_{\beta}^*(\mu,{\hat{O}})&=\inf_{\gamma\in\Gamma} \hat{J}_{\beta}(\mu,\hat{O},\gamma).
% \text{and }J_{\beta}^*(P,\cal{T})&=\inf_{\gamma\in\Gamma} J_{\beta}(P,\cal{T},\gamma)
\end{align*}

The above construction implies the following immediate result.
\begin{assumption}\label{channel_kernel_reg}
We assume that
\begin{itemize}
\item[a.] $|c(x,u) - c(x',u)|\leq K_c \|x-x'\|$ for some $K_c<\infty$ and $\|c\|_\infty <\infty$.
\item[b.] $W_1\left(\mathcal{T}(\cdot|x,u),\mathcal{T}(\cdot|z,u)\right)\leq K_T \|x-z\|$, for all $u \in \mathds{U}$. 
\item[c.] $\|O(\cdot|x)-O(\cdot|z)\|_{TV}\leq K_O \|x-z\|$, (or $\left|g(x,y) - g(z,y)\right|\leq K_O\|x-z\|$)
\end{itemize}
for all $x,z\in\mathds{X}$ for some $K_f<\infty$.
\end{assumption} 

%\sy{Ali: observe that the above implies that the belief-MDP is Wasserstein regular and therefore the results in \cite{demirci2023average} are applicable.}

\begin{lemma}\label{imm_lem}
Under Assumption \ref{channel_kernel_reg}, we have that 
\begin{align*}
&W_1\left(\mathcal{T}(\cdot|x,u), \hat{\mathcal{T}}(\cdot|x,u)\right)\leq K_T L_{\mathds{X}}\\
& \|O(\cdot|x) - \hat{O}(\cdot|x)\|_{TV} \leq K_O L_{\mathds{X}} ,\quad (\left|g(x,y) - \hat{g}(x,y)\right|\leq K_O L_{\mathds{X}})
\end{align*}
where \begin{align}
L_{\mathds{X}}=\max_{i\in \{1,\dots,M\}}\sup_{x,z'\in A_i}\|x-z\|.
\end{align}
\end{lemma}

\begin{proof}
The proof can be found in Appendix \ref{imm_lem_proof}.
\end{proof}

The following result will be critical to prove the main result of this section:

\begin{lemma}\label{iter_lem}
Under Assumption \ref{channel_kernel_reg},  we have that
\begin{align*}
\left|E\left[c(X_t,\gamma(Y_{[0,t]}))\right] - \hat{E}\left[\hat{c}(X_t,\gamma(Y_{[0,t]}))\right]\right| \leq K_O\|c\|_\infty L_\mathds{X} \sum_{n=0}^t \sum_{m=0}^n K_T^m + K_c L_\mathds{X}\sum_{n=0}^t K_T^n
\end{align*}
for any policy $\gamma\in\Gamma$ where $E$ and $\hat{E}$ represent the expectation operators respectively under the models $(\mathcal{T}, O)$ and $(\hat{\mathcal{T}}, \hat{O})$.
\end{lemma}

\begin{proof}
The proof can be found in Appendix \ref{iter_lem_proof}. 
\end{proof}

\begin{prop}\label{key_prop}
Under Assumption \ref{channel_kernel_reg}, if we further assume  that $\beta K_T <1$, we can the write
\begin{align*}
J_\beta(\mu_0,O,\gamma) - \hat{J}_\beta (\mu_0,\hat{O},\gamma) \leq \frac{K_O\|c\|_\infty L_\mathds{X}}{(1-\beta)^2(1-\beta K_T)} + \frac{K_c L_\mathds{X}}{(1-\beta)(1-\beta K_T)}
\end{align*}
for any admissible policy $\gamma\in\Gamma$.
\end{prop}
\begin{proof}
Note that we have
\begin{align*}
&J_\beta(\mu_0,O,\gamma) = \sum_{t=0}^\infty \beta^t E_{\mu_0}^\gamma\left[c(X_t,\gamma(Y_{[0,t]}))\right]\\
&\hat{J}_\beta(\mu_0,\hat{O},\gamma) = \sum_{t=0}^\infty \beta^t \hat{E}_{\mu_0}^\gamma\left[\hat{c}(X_t,\gamma(Y_{[0,t]}))\right]
\end{align*}
Hence the result is a direct application of Lemma \ref{iter_lem}.
\end{proof}

\subsection{Near Optimal Policy Construction Based on Finite State Space}

Recall that the control of the partially observed system can be reduced to the control of a belief MDP. For the belief state based control, the decision maker calculates and applies the optimal policy $\gamma^*(\cdot):\P(\mathds{X})\to \mathds{U}$, by tracking the belief state
 \[\pi_t=Pr(X_t\in\cdot| Y_{[0,t]},U_{[0,t]} ).\]
As a result of discretization of the hidden state space, the agent can construct the conditional probabilities on the finite state space $\hat{\mathds{X}}$, using the observations. At time $t$, the belief state on $\hat{\mathds{X}}$ can be written as
\begin{align}\label{hmm_finite}
 \hat{\pi}_t (i)&= Pr (\hat{X}_t =i| y_{[0,t]},u_{[0,t-1]}) \nonumber\\
&= \frac{\sum_{\hat{x}_{t-1}} \hat{g}(i,y_{t}) \hat{\mathcal{T}} (i|\hat{x}_{t-1},u_{t-1}) \hat{\pi}_{t-1}(\hat{x}_{t-1}) }{\sum_{\hat{x}_{t-1}} \sum_{\hat{x}_t}\hat{g}(\hat{x}_t,y_{t}) \hat{\mathcal{T}} (\hat{x}_t|\hat{x}_{t-1},u_{t-1}) \hat{\pi}_{t-1}(\hat{x}_{t-1})}\nonumber\\
&=: G(\hat{\pi}_{t-1},y_t,u_{t-1})(i)
\end{align}
and thus the decision maker can track the belief state via these iterations. Furthermore, the control policy can be solved for the belief state on the simplex $\P(\mathds{\hat{X}}) \subset \mathds{R}^{|\mathds{\hat{X}}|}$ using the belief MDP that corresponds to the transition and channel model $\hat{\mathcal{T}}$, and $\hat{O}$ (or the channel density $\hat{g}(\hat{x},y)$). The resulting policy $\hat{\gamma}:\P(\hat{\mathds{X}}) \to \mathds{U}$ can also be realized as a mapping from $h_t=\{y_{[0,t]},u_{[0,t-1]}\}$ to $\mathds{U}$ since the belief at any time $t$ is also a function of $h_t$, i.e. 
\begin{align*}
\hat{\gamma}(\hat{\pi}_t) =\hat{\gamma}(Pr(\hat{X}_t\in\cdot|y_{[0,t]},u_{[0,t-1]})).
\end{align*}
One can then directly apply Proposition \ref{key_prop} to find an upper-bound for the performance loss of the approximate control policy:
\begin{theorem}\label{rate_conv}
Under Assumption \ref{channel_kernel_reg}, if we further assume that $\beta K_T <1$, we can the write
\begin{align*}
J_\beta(\mu_0,O,\hat{\gamma}) - J_\beta^*(\mu_0,O)\leq \frac{2K_O}{(1-\beta)^2(1-\beta K_T)} L_{\mathds{X}}+ \frac{2K_c }{(1-\beta)(1-\beta K_T)}L_\mathds{X}
\end{align*}
where $\hat{\gamma}$ is the optimal policy for the control problem based on the discretized hidden state space $\hat{\mathds{X}}$. Note that the first term above represent the realized accumulated cost if one applies the policy based on the the finite space for the true system, whereas the second term is the optimal cost that can be achieved for the correct model under admissible policies.
\end{theorem}
\begin{proof}
We write
\begin{align*}
J_\beta(\mu_0,O,\hat{\gamma}) - J_\beta^*(\mu_0,O) \leq \left|J_\beta(\mu_0,O,\hat{\gamma}) - \hat{J}^*_\beta(\mu_0,\hat{O})\right| + \left| \hat{J}^*_\beta(\mu_0,\hat{O}) -  J_\beta^*(\mu_0) \right|.
\end{align*}
The first term is bounded by Proposition \ref{key_prop} as both costs are induced by the same policy $\hat{\gamma}$. For the second term, we can also use Proposition \ref{key_prop} with the following extra step: if $\hat{J}^*_\beta(\mu_0,\hat{O}) >  J_\beta^*(\mu_0,O) $
\begin{align*}
\left| \hat{J}^*_\beta(\mu_0,\hat{O}) -  J_\beta^*(\mu_0,O) \right|\leq \hat{J}_\beta(\mu_0,\hat{O},\gamma^*) -  J_\beta^*(\mu_0,O) 
\end{align*}
if $\hat{J}^*_\beta(\mu_0,\hat{O}) \leq   J_\beta^*(\mu_0,O) $ we can then write
\begin{align*}
\left| \hat{J}^*_\beta(\mu_0,\hat{O}) -  J_\beta^*(\mu_0,O) \right|\leq  J_\beta(\mu_0,O,\hat{\gamma}) - \hat{J}^*_\beta(\mu_0,\hat{O}) 
\end{align*}
where we used the fact that $J_\beta^*(\mu_0,O)=J_\beta(\mu_0,O,\gamma^*)$ and $\hat{J}_\beta^*(\mu_0,\hat{O})=\hat{J}_\beta(\mu_0,\hat{O},\hat{\gamma})$, and the fact that $\hat{\gamma}$ and $\gamma^*$ are optimal for the corresponding models, so using any other policy increases the cost.  Hence, in either case, we can use Proposition \ref{key_prop} which concludes the proof. 
\end{proof}

\subsection{Asymptotic Analysis}
In this section, we analyze the performance of the hidden state discretization method asymptotically. For this part, rather than choosing an arbitrary quantization scheme, we will work with a particularly constructed quantization scheme. We assume here that $\mathds{X}$ is compact, and for some $n\in\mathds{N}$, we select a finite set (whose existence is guaranteed under the compactness assumption) $\hat{\mathds{X}}_n =\{x_{n,1},\dots,x_{n,k_n}\}$ such that 
\begin{align*}
\min_{1,\dots,k_n} \|x-x_{n,i}\|<1/n, \text{ for all } x\in\mathds{X}.
\end{align*}
That is, $\mathds{\hat{X}}_n$ is a $1/n$ net in $\mathds{X}$. We choose the construction sets $A_i$'s, in a nearest neighbor way such that 
\begin{align*}
\max_{i}\sup_{x,x'\in A_i}\|x-x'\|<1/n.
\end{align*}
We will construct the approximate transition and channel kernels in the same way as (\ref{channel_kernel_reg}). But, we will denote them by $\hat{\mathcal{T}}_n$ and $\hat{O}_n$ to emphasize their dependence on $n\in\mathds{N}$.

Furthermore, instead of using Assumption \ref{channel_kernel_reg} with Lipschitz continuity we will assume the following:
\begin{assumption}\label{channel_kernel_cont}
We assume that
\begin{itemize}
\item[i.] $c(x,u)$ is continuous in $(x,u)$.
\item[ii.] $\mathcal{T}(\cdot|x,u)$ is weakly continuous in $(x,u)$.
\item[iii.] The channel $O(\cdot|x)$ is total variation continuous in $x$. 
\end{itemize}
\end{assumption} 
\begin{remark}
Under the density assumption on the channel $O(\cdot|x)$, we will sometimes equivalently use the assumption that $g(x,y)$ is continuous in $x$.
\end{remark}
The following then is a direct implication of Assumption \ref{channel_kernel_cont}, see also \cite[Theorem 16]{kara2021chapter}.
\begin{lemma}\label{imm_lem2}
Let $\mathds{U}$ be finite. Recall that $\mathds{U}$ is assumed finite, with arbitrarily small loss. Under Assumption \ref{channel_kernel_cont}, we have that
\begin{align*}
\hat{\mathcal{T}}_n(\cdot|x_n,u)\to \mathcal{T}(\cdot|x,u)
\end{align*}
weakly for every $x_n\to x$, and for every $u\in\mathds{U}$. Furthermore
\begin{align*}
\|\hat{O}_n(\cdot|x_n)-  O(\cdot|x)\|_{TV} \to 0
\end{align*}
for every $x_n\to x$, and for every $u\in\mathds{U}$.
\end{lemma}

\begin{theorem}
Under Assumption \ref{channel_kernel_cont}, we can then write
\begin{align*}
J_\beta(\mu_0,O,\hat{\gamma}_n) \to J_\beta^*(\mu_0,O)
\end{align*}
where $\hat{\gamma}_n$ is the optimal policy for the control problem based on the discretized hidden state space $\hat{\mathds{X}}_n$ and models $\hat{\mathcal{T}}_n$ and $\hat{O}_n$. 
\end{theorem}
\begin{proof}
The proof follows similar steps as in the proof of Theorem \ref{rate_conv}, also see \cite[Theorem 4.4]{kara2020robustness} where only transition discrepancy is considered.
\end{proof}

\subsection{A Discussion on the Hidden Space Discretization}
One can construct a belief MDP model based on the discretization of the hidden state space, and using the transition and channel models introduced in (\ref{disc_hidden_model}). This approximation results in a fully observed control problem where the state space  is the simplex on $\hat{\mathds{X}}$ where $|\hat{\mathds{X}}|=M$.  

As a result of this approximation, the belief state on the finite space can be computed iteratively using (\ref{hmm_finite}). Although, the resulting state space is continuous even after the discretization of the hidden space, one can apply further quantization directly on the resulting simplex to find near optimal controllers defined on a finite space; (see \cite[Section 5]{saldi2019pomdp}). Notably, observe that \cite[Theorem 2.6]{kreitmeier2011optimal} relates the Wasserstein error with the quantization diameter.

We also note that the hidden state space discretization indirectly approximates the belief space $\P(\mathds{X})$ as well. In particular, for the finite space $\hat{\mathds{X}}=\{x_1,\dots,x_M\}$, all the probability measures that assign the same measure to the quantization bins, $B_i$'s, form an equivalence class after the discretization of the hidden space. That is if $P,Q\in\P(\mathds{X})$ are such that $P(B_i)=Q(B_i)$ for all $i\in\{1,\dots,M\}$ then we group $P$ and $Q$ together (and all the other measures that have the same property). Hence, the methods in \cite{SaLiYuSpringer,SaYuLi15c} and the Wasserstein and weak continuity of the belief kernel (\cite{KSYWeakFellerSysCont,demirci2023average,demirciRefined2023}) as an alternative approach to establish near optimality.

However, we should recall that this approximation relies on the knowledge of the model. In particular, the belief state on the finite set $\hat{\mathds{X}}$ is computed with the model knowledge. Hence, combining this approximation method with learning methods may not be possible if one only has access to the noisy  observation, $y_{[0,t]}$, and control action, $u_{[0,t]}$, variables for learning period, the estimation of the kernels $\mathcal{T}(\cdot|x,u)$ and $O(\cdot|x)$  is possible if one has access to the hidden state process as well.

Hence, in what follows, we will focus on alternative compression schemes that uses finite memory information variables.

%\sy{Some discussion with the paper with Emre should be made here. Note that space discretization is a method for quantizing probability measures under the Wasserstein distance}.

\section{Finite Memory Based Compression}\label{FiniteMemoryCompSec}
In this section, we will focus on compression schemes using the finite memory information variables. We note that unlike the methods in Section \ref{hidden_state_disc}, we keep the hidden state $\mathds{X}$ as it is, and do not discretize it for the techniques used in this section.
\subsection{An Alternative Finite Window Belief-MDP Reduction}\label{alt_bel_model}
The construction in this section will build on \cite{kara2021convergence}.
 The belief MDP reduction reveals that for the optimal performance, belief state is a sufficient knowledge about the system, hence, one needs to keep track of all the past observation and actions variables. In this section, we will focus on the effect of the finite memory use on the performance decrease as opposed to the full memory of the past. To this end, we construct an alternative fully observed MDP reduction using the predictor from $N$ stages earlier and the most recent $N$ information variables (that is, measurements and actions). 
This new construction allows us to highlight the most recent information variables and {\it compress} the information coming from the past history via the predictor as a probability measure valued variable.

%We also note that the construction and the results presented in this section will be valid for any channel model, as long as the underlying hidden state dynamics stay the same. In particular, the notation and the construction is based on the discrete channel $\hat{O}$.

%For the remainder of the paper, to emphasize the prior distribution of the starting state variable, we will use the following notation for conditional probabilities on state and observation variables.
%\begin{definition}
%Assume that the initial state $X_0$ has a prior distribution $\mu\in\P(\mathds{X})$. Then, for the conditional distribution of $X_t$ given the past observation and action variables $\{y_t,\dots,y_0\}$, $\{u_{t-1},\dots,u_0\}$ we define
%\begin{align*}
%&P^\mu(X_t\in\cdot|y_t,\dots,y_0,u_{t-1},\dots,u_0)\\
%&:=Pr(X_t\in\cdot|y_t,\dots,y_0,u_{t-1},\dots,u_0).
%\end{align*}
%Similarly, for the conditional distribution of observation variables, we define:
%\begin{align*}
%&P^\mu(Y_t\in\cdot|y_{t-1},\dots,y_0,u_{t-1},\dots,u_0)\\
%&:=Pr(Y_t\in\cdot|y_{t-1},\dots,y_0,u_{t-1},\dots,u_0).
%\end{align*}
%\end{definition}
Consider the following state variable at time $t$:
\begin{align}\label{finite_belief_state}
\hat{z}_t=(\pi_{t-N}^-,h_t^N)
\end{align}
where, for $N\geq 1$
\begin{align*}
\pi_{t-N}^-&=P^\mu(X_{t-N}\in \cdot|y_{t-N-1},\dots,y_0,u_{t-N-1},\dots,u_0),\\
h_t^N&=\{y_t,\dots,y_{t-N},u_{t-1},\dots,u_{t-N}\}
\end{align*}
and $h_t^N=y_t$ for $N=0$
with $\mu$ being the prior probability measure on $X_0$. \adk{ Above, we use the notation
\begin{align*}
P^\mu(X_{t-N}\in \cdot|y_{t-N-1},\dots,y_0,u_{t-N-1},\dots,u_0):= Pr(X_{t-N}\in \cdot|y_{t-N-1},\dots,y_0,u_{t-N-1},\dots,u_0)
\end{align*}
which is the posterior distribution of $X_{t-N}$ conditioned on $\{y_{t-N-1},\dots,y_0,u_{t-N-1},\dots,u_0\}$ with a prior distribution $\mu\sim X_0$.}

 The state space with this representation is $\hat{\mathcal{Z}}=\P(\mathds{X})\times \mathds{Y}^{N+1}\times \mathds{U}^{N}$ where we equip $\hat{\mathcal{Z}}$ with the product topology where we consider the weak convergence topology on the $\P(\mathds{X})$ coordinate and the usual (coordinate) topologies on $\mathds{Y}^{N+1}\times \mathds{U}^{N}$ coordinates. 

%This new state representation can be mapped to the belief state $z_t$ defined in (\ref{belief_state}). Consider the map $\psi: \hat{\mathcal{Z}}\to \P(\mathds{X})$, for some $\hat{z}_t=(\pi_{t-N}^-,I_t^N)$
%\begin{align*}
%&\psi(\hat{z}_t)=\psi(\pi_{t-N}^-,I_t^N)=P^{\pi_{t-N}^-}(X_t\in\cdot|I_t^N)\\
%&=P^\mu(X_t\in\cdot|y_t,\dots,y_0,u_{t-1},\dots,u_0)=z_t
%\end{align*}
%such that the map $\psi$ acts as a Bayesian update of $\pi_{t-N}^-$ using $I_t^N$. Using this map, 
We can now define the stage-wise cost function and the transition probabilities. Consider the new cost function $\hat{c}:\hat{\mathcal{Z}}\times \mathds{U}\to \mathds{R}$, using the cost function $\tilde{c}$ of the belief MDP (defined in (\ref{belief_cost})) such that 
\begin{align}\label{hat_cost}
&\hat{c}(\hat{z}_t,u_t)=\hat{c}(\pi_{t-N}^-,h_t^N,u_t)
=\int_\mathds{X}c(x_t,u_t)P^{\pi^-_{t-N}}(dx_t|y_t,\dots,y_{t-N},u_{t-1},\dots,u_{t-N}).
\end{align}
\adk{In particular, noting that 
\begin{align*}
\pi_{t-N}^-&=P^\mu(X_{t-N}\in \cdot|y_{t-N-1},\dots,y_0,u_{t-N-1},\dots,u_0)
\end{align*}
we have 
\begin{align*}
P^{\pi^-_{t-N}}(dx_t|y_t,\dots,y_{t-N},u_{t-1},\dots,u_{t-N}) = P^\mu(dx_t|y_t,\dots,y_0,u_{t-1},\dots,u_0)
\end{align*}
which implies 
\begin{align*}
\hat{c}(\hat{z}_t,u_t)=\hat{c}(\pi_{t-N}^-,h_t^N,u_t)
&=\int_\mathds{X}c(x_t,u_t)P^{\pi^-_{t-N}}(dx_t|y_t,\dots,y_{t-N},u_{t-1},\dots,u_{t-N})\\
&= \int_\mathds{X}c(x_t,u_t)P^{\mu}(dx_t|y_t,\dots,y_{0},u_{t-1},\dots,u_{0}).
\end{align*}}
\adk{
Furthermore, we can define the transition probabilities for $N=1$ as follows: for some $A\in \B(\hat{\mathcal{Z}})$ such that 
\[\ A= B\times\{\hat{y}_{t+1},\hat{y}_{t},\hat{u}_{t}\},\quad  B\in\B(\P(\mathds{X})), \quad \hat{y}_{t+1},\hat{y}_{t},\hat{u}_{t}\in \mathds{Y}^2\times\mathds{U} \]
conditioned on a path of state and action variables $\hat{z}_{t},\dots,\hat{z}_0,u_{t},\dots,u_0$, we write
\begin{align}\label{fin_kernel}
&Pr(\hat{Z}_{t+1}\in A|\hat{z}_{t},\dots,\hat{z}_0,u_{t},\dots,u_0) =Pr(\Pi_{t}^-\in B,{Y}_{t+1}=\hat{y}_{t+1},{Y}_{t}=\hat{y}_{t},{U}_{t}=\hat{u}_{t}|\pi_{[0,t-1]}^-,y_{[0,t]},u_{[0,t]})\nonumber\\
&=\mathds{1}_{\{y_{t}=\hat{y}_{t},u_t=\hat{u}_{t},G(\pi_{t-1}^-,y_{t-1},u_{t-1})\in B\}} P^{\pi_{t-1}^-}(\hat{y}_{t+1}|y_t,y_{t-1},u_{t},u_{t-1})\nonumber\\
&=Pr(\Pi_{t}^-\in B,{Y}_{t+1}=\hat{y}_{t+1},{Y}_{t}=\hat{y}_{t},{U}_{t}=\hat{u}_{t}|\pi_{t-1}^-,y_t,y_{t-1},u_t,u_{t-1})\nonumber\\
&=Pr(\hat{Z}_{t+1}\in A|\hat{z}_t,u_t)=:\int_A\hat{\eta}(d\hat{z}_{t+1}|\hat{z}_t,u_t)
\end{align}
where the map $G$ is defined as 
\begin{align*}
&G(\pi_{t-1}^-,y_{t-1},u_{t-1})=P^\mu(X_{t}\in\cdot|y_{t-1},\dots,y_0,u_{t-1},\dots,u_0).
\end{align*}}

%Hence, $\hat{\eta}$ defines a controlled transition model for the new states $\hat{z}_{t+1}\in \hat{\mathcal{Z}}$. 
Hence, we have a proper fully observed MDP, with the cost function $\hat{c}$, transition kernel $\hat{\eta}$ and the state space $\hat{\mathcal{Z}}$. \adk{ Furthermore, we denote by $\hat{\Gamma}$, the set of admissible policies for the new state model defined in this section.}

%Note that any policy $\phi:\P(\mathds{X})\to \mathds{U}$ defined for the belief MDP, can be extended to the newly defined finite window belief-MDP using the map $\psi$, and defining $\hat{\phi}:=\phi\circ\psi$.
%Thus, if an optimal policy can be found for the belief MDP, say $\phi^*$, the policy $\hat{\phi}^*=\phi^*\circ\psi$ is an optimal policy for the newly defined MDP. 

We now write the discounted cost optimality equation for the newly constructed finite window belief MDP. %Note that with the alternative approach the state $\hat{z}$ can only be written, if we have at least $N$ information variables. Therefore, given that the decision maker observed at least $N$ information variables, we write the following fixed point equation
\begin{align}\label{fin_wind_val}
J^*_\beta(\hat{z})&=\inf_{u\in\mathds{U}}\left(\hat{c}(\hat{z},u)+\beta \int J_\beta^*(\hat{z}_1)\hat{\eta}(d\hat{z}_1|\hat{z},u)\right).
\end{align}
%We can rewrite this fixed point equation in a different form, for notation ease assume $N=1$.
%If $\hat{z}$ has the form $(\pi_0^-,y_1,y_0,u_0)$, then we can rewrite 
%\begin{align}\label{fixed_wind}
%&J_\beta^*(\pi_0^-,y_1,y_0,u_0)\nonumber\\
%&=\min_{u_{1}\in\mathds{U}}\bigg(\hat{c}(\pi_0^-,y_1,y_0,u_0,u_1)+\beta \sum_{y_2\in\mathds{Y}} J_\beta^*(\pi_1^-(\pi_0^-,y_0,u_0),y_{2},y_1,u_{1})P^{\pi_0^-}(y_{2}|y_1,y_0,u_1,u_0)\bigg).
%\end{align}
%This representation will play an important role in the analysis of the problem. Note that the policy $\hat{\phi}^*=\phi^*\circ \psi$ satisfies this fixed point equation.

%The following fixed point equation can also be defined for any policy $\hat{\phi}:\hat{\mathcal{Z}}\to \mathds{U}$
%\begin{align*}
%J_\beta(\hat{z},\hat{\phi})=\hat{c}(\hat{z},\hat{\phi}(\hat{z}))+\beta \int J_\beta(\hat{z}_1,\hat{\phi})\hat{\eta}(d\hat{z}_1|\hat{z},\hat{\phi}(\hat{z}))
%\end{align*}
%where $J_\beta(\hat{z},\hat{\phi})$ denotes the value function under the policy $\hat{\phi}$ for the initial point $\hat{z}$.

This construction results in a fully observed control problem with a state space $\hat{\mathcal{Z}}$ that does not grow over time, and is without loss of optimality. However, the first coordinate of the state variable $\hat{z}_t = (\pi^-_{t-N} , h_t^N)$, which is the predictor at time $t-N$, requires the model knowledge and can be practically hard to track much like the belief state. Therefore, we introduce an approximation method in the following section that maps this predictor to a fixed point and keeps the finite window information variables untouched.

\subsection{Approximation of the Finite Window Belief-MDP}\label{app_section}
The alternative belief MDP model constructed in the previous section lives in the state space
\begin{align*}
\hat{\mathcal{Z}}=\bigg\{&\pi,y_{[0,N]},u_{[0,N-1]}:\pi\in\mathcal{P}(\mathds{X}), y_{[0,N]}\in{\mathds{Y}}^{N+1}, u_{[0,N-1]}\in\mathds{U}^N\bigg\},
\end{align*}
where the first coordinate summarizes the past information, and the second and the last coordinates carry the information from the most recent $N$ time steps.

%hence, all three coordinates take values from uncountable spaces. Our goal in this section is to map these states to an approximate state that takes values from a a finite set. The alternative belief MDP state is by itself contains sufficient information, where the first coordinate summarizes the past information in a probability measure valued variable. The idea of this section will be to map the first coordinate to a fixed probability measure, and hence forget the past in a way, then map the continuous valued observation and actions from the most recent $N$ time steps the a finite subset through discretization and nearest neighbor mapping.
Consider the following set $\hat{\mathcal{Z}}_{\pi^*}^N$ for a fixed $\pi^*\in\P(\mathds{X})$
\begin{align*}%\label{finite_set_belief}
\hat{\mathcal{Z}}_{\pi^*}^N=\bigg\{\pi^*,y_{[0,N]},u_{[0,N-1]}: y_{[0,N]}\in{\mathds{Y}}^{N+1}, u_{[0,N-1]}\in\mathds{U}^N\bigg\}
\end{align*}
such that the state at time $t$ is $\hat{z}_t^N=(\pi^*,h_t^N)$. Compared to the state $\hat{z}_t=(\pi_{t-N}^-,h_t^N)$ defined in (\ref{finite_belief_state}), this approximate model uses $\pi^*$ as the predictor, no matter what the real predictor at time $t-N$ is.

The cost function is defined as 
\begin{align}\label{cost_window}
&\hat{c}(\hat{z}^N_t,u_t)=\hat{c}(\pi^*,h_t^N,u_t)=\int_\mathds{X}c(x_t,u_t)P^{\pi^*}(dx_t|y_t,\dots,y_{t-N},u_{t-1},\dots,u_{t-N}).
\end{align}
\adk{
We define the controlled transition model  for some $\hat{z}_{t+1}^N=(\pi^*,h_{t+1}^N)$ and $\hat{z}_{t}^N=(\pi^*,h_{t}^N)$
\begin{align}\label{eta_N}
&\hat{\eta}^N(\hat{z}_{t+1}^N|\hat{z}_t^N,u_t)=\hat{\eta}^N(\pi^*,h_{t+1}^N|\pi^*,h_t^N,u_t):=\hat{\eta}\bigg(\P(\mathds{X}),h_{t+1}^N|\pi^*,h_t^N,u_t\bigg)
\end{align}
where $\hat{\eta}$ is the transition kernel of the alternative finite window belief MDP reduction, see (\ref{fin_kernel}). In words,  starting from the state $\hat{z}_t^N=(\pi^*,h_t^N)$ under the action $u_t$, for any possible future state under the dynamics of the true world (i.e. $\hat{\eta}$), we apply a crude quantization for the first coordinate and map any $\pi_{t-N+1}^-\in\P(\mathds{X})$ to $\pi^*$. 

For simplicity, if we assume $N=1$, then the transitions can be rewritten for some $h_{t+1}^N=(\hat{y}_{t+1},\hat{y}_t,\hat{u}_t)$ and $h_t^N=(y_t,y_{t-1},u_{t-1})$
\begin{align}\label{eta_N_1}
\hat{\eta}^N(\pi^*,\hat{y}_{t+1},\hat{y}_t,\hat{u}_t|\pi^*,y_t,y_{t-1},u_{t-1},u_t)&=\hat{\eta}(\P(\mathds{X}),\hat{y}_{t+1},\hat{y}_t,\hat{u}_t|\pi^*,y_t,y_{t-1},u_{t-1},u_t)\nonumber\\
&=\mathds{1}_{\{y_t=\hat{y}_t,u_t=\hat{u}_t\}}P^{\pi^*}(\hat{y}_{t+1}|y_t,y_{t-1},u_t,u_{t-1}).
\end{align} 
%Note that the transition model can equivalently written as
%\begin{align}\label{alt_eta_N}
%&\hat{\eta}^N(\hat{z}_{t+1}^N|\hat{z}_t^N,u_t)=\int_{\mathds{X}}\hat{O}(y_{t+1}|x_{t+1})P^{\pi^*}(dx_{t+1}|y_t,\dots,y_{t-N},u_{t-1},\dots,u_{t-N}).
%\end{align}

%For simplicity, if we assume $N=1$, then the transitions can be rewritten for some $I_{t+1}^N=(\hat{y}_{t+1},\hat{y}_t,\hat{u}_t)$ and $I_t^N=(y_t,y_{t-1},u_{t-1})$
%\begin{align}\label{eta_N_1}
%\hat{\eta}^N(\pi^*,\hat{y}_{t+1},\hat{y}_t,\hat{u}_t|\pi^*,y_t,y_{t-1},u_{t-1},u_t)&=\hat{\eta}(\P(\mathds{X}),\hat{y}_{t+1},\hat{y}_t,\hat{u}_t|\pi^*,y_t,y_{t-1},u_{t-1},u_t)\nonumber\\
%&=\mathds{1}_{\{y_t=\hat{y}_t,u_t=\hat{u}_t\}}P^{\pi^*}(\hat{y}_{t+1}|y_t,y_{t-1},u_t,u_{t-1}).
%\end{align} 

%Since, the predictor at time $t-N$ is fixed to $\pi^*$, we will now rewrite the state, cost and transition models for the approximate model in an alternative but an equivalent way to make notation more clear.

We denote the optimal value function for the approximate model by $J_\beta^N$. We also denote by $\phi^N$ an optimal policy for the resulting approximate model.} %we can write the following fixed point equation
%\begin{align}\label{N_fixed}
%J_\beta^N(\hat{z}^N)=\min_{u\in\mathds{U}}\bigg(&\hat{c}(\hat{z}^N,u)\\
%&+\beta\int_{\hat{z}_1^N\in\hat{\mathcal{Z}}^N_{\pi^*}}J_\beta^N(\hat{z}_1^N)\hat{\eta}^N(d\hat{z}_1^N|\hat{z}^N,u)\bigg).\nonumber
%\end{align}
Note that both $J^N_\beta$ and $\phi^N$ are defined on the set $\hat{\mathcal{Z}}^N_{\pi^*}$. However, we can simply extend them to the set $\hat{\mathcal{Z}}$ by defining for any $\hat{z}=(\pi,y_1,y_0,u_0)\in\hat{\mathcal{Z}}$ (assuming $N=1$)
\begin{align*}
\tilde{J}^N_\beta(\hat{z})=\tilde{J}^N_\beta(\pi,y_1,y_0,u_0)&:=J^N_\beta(\pi^*,y_1,y_0,u_0)\\
\tilde{\phi}^N(\hat{z})=\tilde{\phi}^N(\pi,y_1,y_0,u_0)&:=\phi^N(\pi^*,y_1,y_0,u_0).
\end{align*}
Another interpretation of the extension is the following: Since the predictor $\pi^*$ is fixed $\phi^N$ can be thought as a map from $h_t^N$ to $\mathds{U}$, hence it maps the most recent $N$ information variables to the control actions. At any given time, one can simply choose the control action by looking at the most recent $N$ information variables according to $\phi^N$ and ignore the past. See Figure \ref{FigFiniteWMDP} for a comparison of finite memory belief reduction
\begin{figure}
\centering
\epsfig{figure=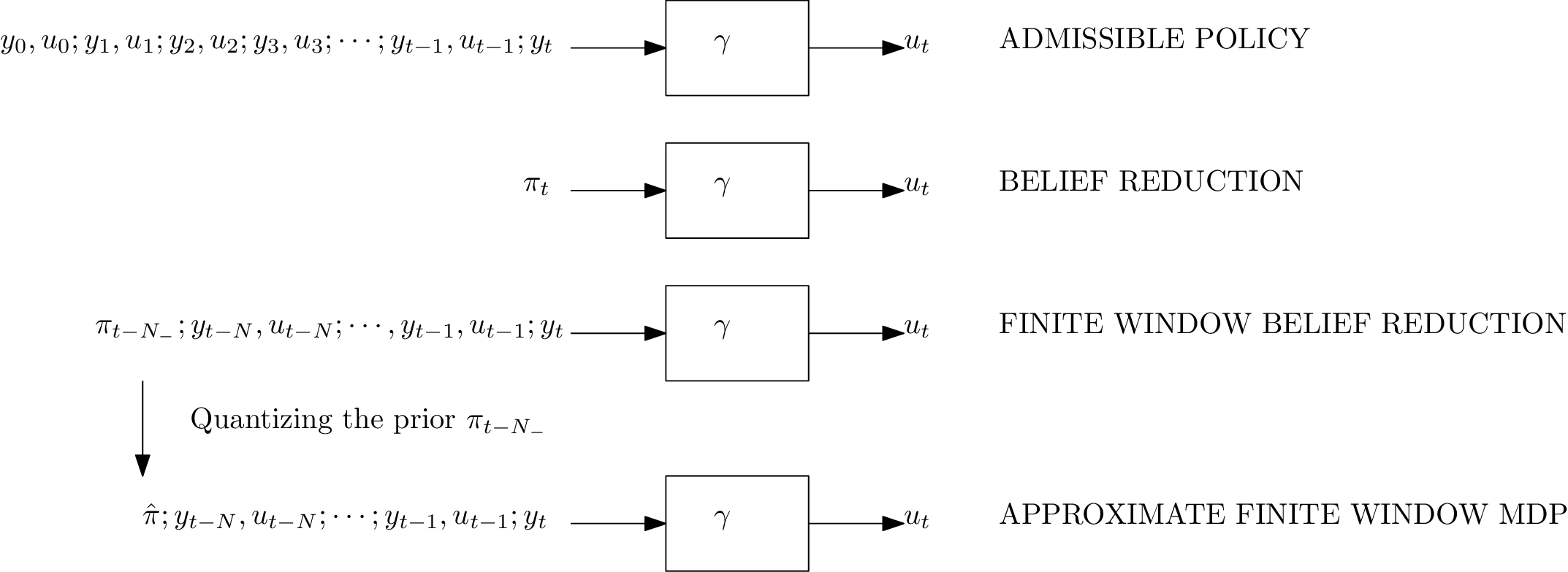,height=5.5cm,width=12cm}
\caption{Construction of the Finite-Window Approximate MDP from the Finite-Window Belief-MDP\cite{kara2021convergence}. The quantization of the finite window MDP model leads to the collapse of the first coordinate to a fixed measure.}\label{FigFiniteWMDP}
\end{figure}

%Before the result, we introduce the following definition and notation.
%\begin{definition}
%For probability measures $\mu,\nu \in \mathcal{P}(\mathds{X})$, the \emph{total variation} metric is given by
  %\begin{align*}
   % \|\mu-\nu\|_{TV}&=\sup_{f:\|f\|_\infty \leq 1}\left|\int f(x)\mu(\dd x)-\int f(x)\nu(\dd x)\right|,
 % \end{align*}
  %\noindent where the supremum is taken over all measurable real $f$ such that $\|f\|_\infty=\sup_{x\in\mathds{X}}|f(x)|\leq 1$. 
%\end{definition}
We define the following constant:
\begin{align}\label{loss_constant}
L_t:=\sup_{\hat{\gamma}\in\hat{\Gamma}}E_{\pi_0^-}^{\hat{\gamma}}\bigg[\|P^{\pi_t^-}(X_{t+N}\in\cdot|Y_{[t,t+N]},U_{[t,t+N-1]})
-P^{\pi^*}(X_{t+N}\in\cdot|Y_{[t,t+N]},U_{[t,t+N-1]})\|_{TV}\bigg]
\end{align} 
which is the expected bound on the total variation distance between the posterior distributions of $X_{t+N}$ conditioned on the same observation and control action variables $Y_{[t,t+N]},U_{[t,t+N-1]}$ when the prior distributions of $X_{t}$ are given by $\pi_t^-$ and $\pi^*$.  \adk{Above, $\hat{\Gamma}$ denote the set of admissible policies for the alternative finite window based state model defined in Section \ref{alt_bel_model}.} %The expectation is with respect to the random realizations of $\pi_t^-$ and $Y_{[t,t+N]},U_{[t,t+N-1]}$ under the true dynamics of the system when the prior distribution of $x_0$ is given by $\pi_0^-$.  This constant represents the bound on the distance of two processes with different starting points when they are updated with the same observation and action processes under same policy.

\begin{prop}\cite[Theorem 3.3]{kara2021convergence}\label{cont_bound}
Let $X_0\sim\mu$ and consider $H_0^N=\{Y_0,\dots,Y_N,U_0,\dots,U_{N-1}\}$  with a policy $\hat{\gamma}$ acting on the first $N$ steps. 
For $\hat{Z}_0=(\mu,H_0^N)$,
\begin{align*}
E_{\mu}^{\hat{\gamma}}\left[\left|J_\beta(\hat{Z}_0,\tilde{\phi}^N) -J^*_\beta(\hat{Z}_0)\right|\right]\leq  \frac{2\|c\|_\infty }{(1-\beta)}\sum_{t=0}^\infty\beta^tL_t
\end{align*}
where $J_\beta(\hat{Z}_0,\tilde{\phi}^N)$ is the cost under the finite window policy $\tilde{\phi}^N$, and $J^*_\beta(\hat{Z}_0)$ is the optimal value function for the initial point $\hat{Z}_0$ (see (\ref{fin_wind_val})). Furthermore, the expectation is with respect to the random realizations of $H_0^N$.
\end{prop}

The finite memory approximation method we have presented in this section results in a compression scheme with a loss of optimality since the information prior to time $t-N$ is ignored. This loss, on the other hand, is controllable with the term $L_t$ (\ref{loss_constant}). Furthermore, the state space of this approximation is $\mathds{Y}^N\times \mathds{U}^{N-1}$ where $N$ is fixed and thus the state space does not grow over time. To track this finite memory state, the controller do not have to perform an extra compression step and can simply use the finite memory variables as they are. 

However, for continuous observation spaces, the finite memory approximation will still not be leading to a finite model. To this end, in the next section, we will focus on the discretization of the observation space. We first focus on the observation discretization without the finite memory requirement, and then we combine the observation discretization with the finite memory method.

\section{Quantization of Observations and Loss in Performance}\label{obs_disc}
The belief MDP reduction reveals that for the optimal performance, belief state is a sufficient knowledge about the system That is, at any given time $t$, the controller can use
\begin{align*}
z_t:=Pr\{X_{t} \in \,\cdot\, | Y_0,\ldots,Y_t, U_0, \ldots, U_{t-1}\} \in \P(\sX). 
\end{align*}
Hence, one needs to keep track of all the past observation and control action variables. Even when the memory is not a constraint, computation of the belief state is not trivial as the observation space is assumed to be continuous. 

We first focus on the discretization of the observation variables. For the discretization, we choose disjoint subsets $\{B_i\}_{i=0}^M\subset\mathds{Y}$ such that $\cup_iB_i=\mathds{Y}$. We define a finite set 
\begin{align*}
\mathds{Y}_M:=\{y_0,\dots,y_M\}.
\end{align*}
Furthermore, the discretization map $\psi_{\mathds{Y}}:\mathds{Y}\to \hat{\mathds{Y}}_M$ is defined such that 
\begin{align}\label{quant_map}
\psi_{\mathds{Y}}(y)=y_i, \text{ if } y\in B_i.
\end{align}
In other words, the discretization checks what bin or set $y$ belongs to and maps it to the representative element of that bin. Selection of the disjoint sets depends on the application and the control problem.

Using the discretization procedure, one can define a further POMDP model, which uses a channel, say $\hat{O}(\cdot|x)$, defined on the finite space $\hat{\mathds{Y}}_M$. Note that this channel is different than the one introduced in (\ref{disc_hidden_model}), but we use the same notation. For any $y_i\in\hat{\mathds{Y}}_M$
\begin{align}\label{deg_channel}
\hat{O}(y_i|x)=O(B_i|x).
\end{align}
This channel is clearly degraded with respect to the original channel, since $y_i$ is conditionally independent of $x$ given the original observation $y$.

Therefore, for prior distribution on $x_0$ given by $\mu$, one can define the optimal cost under the discrete observations by 
\begin{align*}
J_\beta^*(\mu,\hat{O})=\inf_{\hat{\gamma}\in\hat{\Gamma}}J_\beta(\mu,\hat{O},\gamma),
\end{align*}
where the policies, $\hat{\gamma}\in\hat{\Gamma}$ are measurable  with respect to the $\sigma$-algebra
generated by the new information variables where the observations take values from the finite set $\hat{\mathds{Y}}_M$.

\subsection{Effect of Observation Discretization}\label{finite_obs}
In this section, we study the effect of observation discretization on the optimal value. Recall that the original channel is modeled as a stochastic kernel and denoted by $O(dy|x)$. Furthermore, we assume that this channel admits a density function with respect to a reference measure $\lambda(dy)$ such that
\begin{align*}
O(dy|x)=g(x,y)\lambda(dy).
\end{align*}

The observation space discretization is introduced in Section \ref{obs_disc}, and the resulting channel is denoted by $\hat{O}$, which is defined on the finite space $\hat{\mathds{Y}}_M$ such that for any $y_i\in\hat{\mathds{Y}}_M$
\begin{align*}
\hat{O}(y_i|x)=O(B_i|x).
\end{align*}
where $B_i$'s are the quantization bins (or the aggregate sets). We will further assume that the observation space {$\mathds{Y}\subset \mathds{R}^n$} is compact. Our goal in this section is to find bounds for the term
\begin{align*}
\inf_{\hat{\gamma}\in\hat{\Gamma}}J_\beta(\mu,\hat{O},\hat{\gamma})-\inf_{\gamma\in\Gamma}J_\beta(\mu,O,\gamma)
\end{align*}
where $\hat{\Gamma}$ is the set of policies that use the discretized observations, and $\Gamma$ is the policies that use the original observations, and $\mu$ is the prior distribution of $x_0$. Furthermore, the first term is greater as $\hat{O}$ is a degraded channel in comparison to $O$.

We define the following notation, denoting the uniform error bound due to the discretization of the observations
\begin{align}
L_{\mathds{Y}}=\max_{i\in \{1,\dots,M\}}\sup_{y,y'\in B_i}|y-y'|.
\end{align}
We now state the assumptions and the result formally.
\begin{assumption}\label{main_assmp2}
\begin{itemize}
\item $O(dy|x)=g(x,y)\lambda(dy)$, and $g(y,x)$ is continuous in $y$ for every $x\in\mathds{X}$.
\item Stage-wise cost function $c(x,u)$ is bounded such that $\sup_{x,u}c(x,u):=\|c\|_\infty<\infty$.  
\end{itemize}
\end{assumption}
\begin{assumption}\label{main_assmp}
\begin{itemize}
\item {$\mathds{Y}\subset\mathds{R}^n$} is compact.
\item $O(dy|x)=g(x,y)\lambda(dy)$, and $g(y,x)$ is Lipschitz in $y$, such that $|g(x,y)-g(x,y')|\leq\alpha_{\mathds{Y}} |y-y'|$ for every $y,y'\in\mathds{Y}$ and $x\in\mathds{X}$ for some $\alpha_{\mathds{Y}}<\infty$.
\item Stage-wise cost function $c(x,u)$ is bounded such that $\sup_{x,u}c(x,u):=\|c\|_\infty<\infty$.  
\end{itemize}
\end{assumption}

\begin{theorem}\label{disc_obs2}
Under Assumption \ref{main_assmp2},
\begin{align*}
\lim_{L_\mathds{Y}\to 0}\left(J_\beta^*(\mu,\hat{O})-J_\beta^*(\mu,O)\right)=0.
\end{align*}
where $J_\beta^*(\mu,\hat{O}):=\inf_{\hat{\gamma}\in\hat{\Gamma}}J_\beta(\mu,\hat{O},\hat{\gamma})$ and $J_\beta^*(\mu,O):=\inf_{\gamma\in\Gamma}J_\beta(\mu,O,\gamma)$.
\end{theorem}

\begin{theorem}\label{disc_obs}
Under Assumption \ref{main_assmp},
\begin{align*}
J_\beta^*(\mu,\hat{O})-J_\beta^*(\mu,O)\leq \frac{\beta}{(1-\beta)^2}\|c\|_\infty \alpha_{\mathds{Y}} L_{\mathds{Y}}
\end{align*}
\end{theorem}

\begin{proof}
We will only prove Theorem \ref{disc_obs}, the proof for Theorem \ref{disc_obs2} follows from almost identical steps.

We start by introducing a new channel $O'(dy|x)$ acting on the original observation space $\mathds{Y}$. The new channel is conditionally independent of everything else given the output of the discrete channel $\hat{O}$, and it gives a uniformly distributed (with respect to the measure $\lambda(dy)$) random value from $B_i$ if the output of the finite channel $\hat{O}$ is $y_i$. %In other words, given some $y_i\in\hat{\mathds{Y}}_M$, where $y_i$ is the representative for the set $B_i\subset\mathds{Y}$, $O'$ gives a value in $B_i$ according to the uniform distribution on $B_i$. For example, for some set $A\subset \mathds{Y}$ which can be written as $(A\cap B_i)\cup(A\cap B_j)$ for some $B_i,B_j$, we have
%\begin{align*}
%&O'(A|x)=O'(A\cap B_i|x)+O'(A\cap B_j|x)\\
%&=\sum_{k=1}^M (Pr(A\cap B_i|y_k)+Pr(A\cap B_j|y_k))\hat{O}(y_k|x)\\
%&=\frac{\lambda(A\cap B_i)}{\lambda(B_i)}\hat{O}(y_i|x)+\frac{\lambda(A\cap B_j)}{\lambda(B_j)}\hat{O}(y_j|x).
%\end{align*}
Hence, the new channel $O'$ admits a density function, say $g'$, with respect to $\lambda$, such that for some $y' \in B_i$
\begin{align*}
g'(x,y')=\frac{O(B_i|x)}{\lambda(B_i)}=\frac{1}{\lambda(B_i)}\int_{B_i}g(x,y)\lambda(dy).
\end{align*}
Note also that, since we assume that $\mathds{Y}$ is compact, the subsets $B_i$'s are bounded and the uniform distribution is well defined on these sets. Using the assumption that $g(x,y)$ is Lipschitz in $y$, we can conclude that 
\begin{align}\label{chan_lip}
|g'(x,y)-g(x,y)|\leq \alpha_{\mathds{Y}} \max_{i}\sup_{y,y'\in B_i}|y-y'|\leq \alpha_{\mathds{Y}} L_{\mathds{Y}}.
\end{align}
This new channel is introduced only for mathematical convenience, in order to define the degraded POMDP on the same observation space as the original POMDP model. Furthermore, since the only information one can infer about the hidden state $x$, is through the discrete observations $y_i$, the channel $O'$ is a garbling of the channel $\hat{O}$. Hence, it is a direct consequence of Blackwell's informativeness theorem \cite{blackwell1953equivalent} that
\begin{align*}
J_\beta^*(\mu,O')\geq J_\beta(\mu,\hat{O}).
\end{align*}
In fact, one can show that the costs are indeed equal, however, we will only use the inequality as it is sufficient for our purpose. Furthermore, since $O'$ sends values from the original observation space $\mathds{Y}$, the POMDP model with $O'$ also uses the policies from the set $\Gamma$. Therefore, denoting the optimal policy designed for the original channel $O$ by $\gamma^*$, if we apply this policy for the channel $O'$, we get a larger cost, i.e. $J_\beta(\mu,O',\gamma^*)\geq J_\beta^*(\mu,O')$. Hence, for the rest of the proof, we will assume that both models use the same policy.

%We now define {\it strategic measures}  (\cite{Schal}) as the set of probability measures induced on the product spaces of the state and action pairs by admissible control policies: Given an initial distribution on the state, and a policy, one can uniquely define a probability measure on the infinite product space consistent with finite dimensional distributions, by the Ionescu-Tulcea theorem \cite[Proposition C.10]{HernandezLermaMCP}. Define a strategic measure under a policy $\gamma= \{\gamma_0,\gamma_1, \cdots, \gamma_k,\cdots\}$ as a probability measure defined on ${\cal B}(\mathds{X} \times \mathds{Y} \times \mathds{U})^{\mathds{Z}_+}$ by
%\begin{align*}
%&P^{\gamma^n}_{O}(\dd(x_0,y_0,u_0),\dd(x_1,y_1,u_1),\cdots) \nonumber \\
%& \quad = P(\dd x_0) O(\dd y_0|x_0) 1_{\{\gamma(y_0) \in \dd u_0\}} T(\dd x_1|x_0,u_0)\\
%&\quad\qquad O(\dd y_1|x_1) 1_{\{\gamma^n(y_0,y_1) \in \dd u_1\}} \cdots.
%\end{align*}
%Then, with $P^\gamma_O$ and $P^\gamma_{O'}$ denoting the strategic measures for two chains with a policy $\gamma$, and channels $O$ and $O'$ , we have
%\begin{align*}
%&|J_{\beta}( \mu,O',\gamma) - J_{\beta}( \mu,O,\gamma)|\\
%&\leq \sum_k \beta^k \|c\|_\infty \|P^\gamma_{ O'}(d(x,y,u)_{[0,k]})-P^\gamma_{ O}(d(x,y,u)_{[0,k]})\|_{TV}.
%\end{align*}

By \cite[Theorem 6.2]{YukselOptimizationofChannels} we establish the following relation:% between the total variation distance of the strategic measures of the chains until time stage $k$ and the total variation distance of transition kernels.
\begin{align*}
&|J_{\beta}( \mu,O',\gamma) - J_{\beta}( \mu,O,\gamma)|\\
&\leq \sum_k \beta^k \|c\|_\infty k  \sup_{x \in \mathds{X}} \|O'(.|x)- O(.|x)\|_{TV}.
\end{align*}
Furthermore, using (\ref{chan_lip}), we can write that 
\begin{align*}
\sup_{x \in \mathds{X}} \|O'(.|x)- O(.|x)\|_{TV}\leq \alpha_\mathds{Y} L_{\mathds{Y}},
\end{align*}
which concludes the proof.
\end{proof}

\section{Near Optimality of Finite Memory Feedback with Approximated Discrete Observations}\label{FiniteMemDiscreteObs}
Although, the observations are discretized, the resulting belief state still takes values from the space of probability measures and the sufficient information for the optimal control of the POMDP with the finite observation space, requires the controller to keep track of all the history. Therefore, as the next step of the approximation, we only focus on a finite memory of the information variables. We restrict the controller to use the new information variables
\begin{align*}
\hat{H}_t^N=\{\hat{Y}_{[t-N,t]},U_{[t-N,t-1]}\},
\end{align*}
where $\hat{Y}_t$ takes values from the finite set $\hat{\mathds{Y}}_M$. If we denote the set of policies that use $\hat{H}_t^N$, by $\hat{\Gamma}^N$, the question we ask is the following: can we find bounds on
\begin{align}\label{problem}
\inf_{\gamma^N\in\hat{\Gamma}^N}J_\beta(\mu,\hat{O},\gamma^N)-J_\beta^*(\mu,O)
\end{align}
that is what do we loose if we use a finite memory of discretized observations in comparison the optimal cost of the original POMDP with continuous observations. This problem also motivates the natural, though interesting, question: if one has a computational constraint, is it better to use a finer discretization of the observation space, or use a longer memory by keeping the dicretizaton rate the same. 

To analyze the problem, we first discretize the observations. We then analyze the effect of finite memory with the discretized observations. On the effect of the finite memory on the performance, in Section \ref{app_section}, we have presented bounds on the loss if one only uses a fixed length of most recent information variables. We remind that the presented results are be valid for any channel model including the degraded channel model resulting from the observation discretization.

%In Section \ref{finite_obs}, we focus on the effect of the observation discretization. In particular, we analyze the difference between the optimal costs of two POMDP models; one using a channel with continuous observations, and the other one using a degraded channel by dicretizing the original observations.

%\section{Effect of Finite Memory}\label{finite_memory}
%We focus on the effect of the finite memory use on the performance. The method is taken mostly from \cite{kara2021convergence}, however, we will present the method in detail for completeness.

\subsection{Combined Approximation Algorithm and the Near Optimality of the Approximate Policies Under Filter Stability}\label{app_alg}
In this section, we introduce the approximation methodology, which involves a combination of observation space discretization and use of finite memory. We then provide an upper bound on the performance loss when the approximation policy is applied in the original POMDP model, in terms of the approximation parameters. 

The approximation algorithm is as follows:
\begin{itemize}
\item[1. ] Pick disjoint subsets $\{B_i\}_{i=1}^M\subset \mathds{Y}$ such that $\cup_{i}B_i=\mathds{Y}$. Furthermore, pick the representative observation set $\hat{\mathds{Y}}_M=\{y_1,\dots,y_M\}$, and the discretization map $\psi_{\mathds{Y}}:\mathds{Y}\to \hat{\mathds{Y}}_M$ is defined such that 
\begin{align*}
\psi_{\mathds{Y}}(y)=y_i, \text{ if } y\in B_i.
\end{align*}
\item[2. ] Define a new channel $\hat{O}(\cdot|x)$ on the finite set $\hat{\mathds{Y}}_M$ such that for any $y_i\in\hat{\mathds{Y}}_M$
\begin{align*}
\hat{O}(y_i|x)=O(B_i|x).
\end{align*}
\item [3. ] Pick a window length $N$, and define a fully observed MDP, with state space $\hat{\mathds{Y}}_M^N\times\mathds{U}^{N-1}$, the cost function is defined by (\ref{cost_window}), and the transition model defined by (\ref{eta_N}) or (\ref{eta_N_1}). Note that, to define the cost and the transitions, the new channel $\hat{O}$ is used.
 \item [4.] Solve the optimality equation for the finite fully observed MDP, to construct a policy $\hat{\gamma}^N$.
\end{itemize}

\begin{theorem}\label{main_result2}

Suppose Assumption \ref{main_assmp2} holds and  the belief kernel $\eta$ is weakly continuous (\cite{FeKaZa12, KSYWeakFellerSysCont}).
Let $X_0\sim\mu$ and consider $H_0^N=\{Y_0,\dots,Y_N,U_0,\dots,U_{N-1}\}$  with a policy $\hat{\gamma}$ acting on the first $N$ steps. 
For $\hat{Z}_0=(\mu,H_0^N)$ (recall (\ref{finite_belief_state})),
if the policy found after the steps above, $\hat{\gamma}^N$ acts after the first $N$ steps
\begin{align*}
&\lim_{L_\mathds{Y}\to 0}E_{\mu}^{\hat{\gamma}}\left[J_\beta(\hat{Z}_0,\hat{\gamma}^N) -J^*_\beta(\hat{Z}_0)\right]\leq  \frac{2\|c\|_\infty }{(1-\beta)}\sum_{t=0}^\infty\beta^tL_t
\end{align*}
where $L_t$ is the filter stability term for the original channel $O$ and the expectation is with respect to the random realizations of $H_0^N$.

%For the policy found after the steps above, $\hat{\gamma}^N$, if the belief kernel $\eta$ is weakly continuous (\cite{FeKaZa12, KSYWeakFellerSysCont}), we can write that, for any prior distribution $\mu\in\P(\mathds{X})$ 
%\begin{align*}
%&\lim_{L_\mathds{Y}\to 0}\left(J_\beta(\mu,O,\hat{\gamma}_N)-J_\beta^*(\mu,O)\right)\leq \frac{2\|c\|_\infty}{(1-\beta)}\sum_{t=0}^\infty \beta^t L_t,
%\end{align*}
%where $L_t$ is the filter stability term for the original channel $O$.
\end{theorem}
\begin{theorem}\label{main_result}
Suppose Assumption \ref{main_assmp} holds.
Let $X_0\sim\mu$ and consider $H_0^N=\{Y_0,\dots,Y_N,U_0,\dots,U_{N-1}\}$  with a policy $\hat{\gamma}$ acting on the first $N$ steps. 
For $\hat{Z}_0=(\mu,H_0^N)$ (recall (\ref{finite_belief_state})),
if the policy found after the steps above, $\hat{\gamma}^N$ acts after the first $N$ steps
\begin{align*}
&E_{\mu}^{\hat{\gamma}}\left[J_\beta(\hat{Z}_0,\hat{\gamma}^N) -J^*_\beta(\hat{Z}_0)\right]\leq   \frac{\beta}{(1-\beta)^2}\|c\|_\infty \alpha_{\mathds{Y}} L_\mathds{Y}+ \frac{2\|c\|_\infty}{(1-\beta)}\sum_{t=0}^\infty \beta^t L_t,
\end{align*}
where
\begin{align*}
 &L_{\mathds{Y}}:= \max_{i}\sup_{y,y'\in B_i}|y-y'|,\\
L&_t:=\sup_{\hat{\gamma}\in\hat{\Gamma}}E_{\mu}^{\hat{\gamma},\hat{O}}\bigg[\|P^{\pi_t^-}(X_{t+N}\in\cdot|\hat{Y}_{[t,t+N]},U_{[t,t+N-1]})\nonumber\\
&\qquad\qquad\qquad\qquad-P^{\pi^*}(X_{t+N}\in\cdot|\hat{Y}_{[t,t+N]},U_{[t,t+N-1]})\|_{TV}\bigg]
\end{align*}
and $\alpha_{\mathds{Y}}$ is the Lipschitz constant of the density function $g$ of the channel $O$.
\end{theorem}

\begin{proof}[For Theorem \ref{main_result2} and \ref{main_result}]
We first consider the marginal distribution of $X_2$:
\begin{align*}
\mu_2(\cdot) = \int\mathcal{T}(\cdot|x_1,\hat{\gamma}(y_1,y_0))O(dy_1|x_1)\mathcal{T}(dx_1|x_0,\hat{\gamma}(y_0))O(dy_0|x_0)\mu(dx_0).
\end{align*}
Note that we have
\begin{align*}
E_{\mu}^{\hat{\gamma}}\left[J^*_\beta(\hat{Z}_0)\right]=J_\beta^*(\mu_2,O)
\end{align*}
We write
\begin{align*}
E_{\mu}^{\hat{\gamma}}\left[J_\beta(\hat{Z}_0,\hat{\gamma}^N) -J^*_\beta(\hat{Z}_0)\right]&\leq E_{\mu}^{\hat{\gamma}}\left[J_\beta(\hat{Z}_0,\hat{\gamma}^N) -J^*_\beta(\mu_2,\hat{O})\right] + E_{\mu}^{\hat{\gamma}}\left[J_\beta^*(\mu_2,\hat{O})-J^*_\beta(\hat{Z}_0)\right]\\
& =E_{\mu}^{\hat{\gamma}}\left[J_\beta(\hat{Z}_0,\hat{\gamma}^N) -J^*_\beta(\mu_2,\hat{O})\right] + E_{\mu}^{\hat{\gamma}}\left[J_\beta^*(\mu_2,\hat{O})-J^*_\beta(\mu_2,O)\right]
\end{align*}
Note that $\hat{\gamma}^N$ is the finite window policy for the channel $\hat{O}$. The first difference is bounded by Proposition \ref{cont_bound}. Therefore, the first bound is related to the controlled filter stability under the discrete channel $\hat{O}$. For the second term, we use Theorem \ref{disc_obs}.

For the proof of Theorem \ref{main_result2}, the second term goes to $0$, by Theorem \ref{disc_obs2}. For the first term, we again use Proposition \ref{cont_bound} with the fact that the filter stability term under channel $\hat{O}$ converges to the filter stability term under the channel $O$ if the belief kernel $\eta$ is weak Feller.
\end{proof}

\subsubsection{A Discussion on Theorem \ref{main_result}}
One can observe that the first term in the upper bound presented in Theorem \ref{main_result}, is only related to the observation discretization and is not affected by the finite memory use. The result indicates that higher Lipschitz constants will lead to higher performance improvements as the dicretization rate grows. Clearly, if the discretization rate is already high, increasing it further will not make the same impact. For example, for a bounded observation space $\mathds{Y}$, if the size of the chosen finite set is $M$, $L_\mathds{Y}$ can be approximated by $\frac{1}{M}$ under uniform quantization. Thus, increasing the size to $M+1$ would result in 
\begin{align*}
K \alpha_{\mathds{Y}} \frac{1}{M}-K \alpha_{\mathds{Y}} \frac{1}{M+1}=K \alpha_{\mathds{Y}} \frac{1}{M(M+1)}
\end{align*}
decrease on the first term of the upper bound where $K=\frac{\beta\|c\|_\infty}{(1-\beta)^2}$.

 The second term is related to the controlled filter stability under the discrete channel. This term is affected by both the window length $N$, and the rate of the discretization. Increasing the window length will clearly decrease the second term. The better discretization rate will change the second term, since how informative the observations are will also impact the filter stability. One can perform simulations to study the exact impact of the discretization on the controlled filter stability. In the following, we will provide some upper bounds for the second term, using Dobrushin coefficient (Definition \ref{dob_def}) of the original channel $O$.   

\subsubsection{Further Upper-bounds for Filter Stability via Dobrushin Coefficient}\label{dob_sec}

In this section, we discuss the term ${\it L_t}$ term defined in (\ref{loss_constant}). %as

\begin{definition}\cite[Equation 1.16]{dobrushin1956central}\label{dob_def}
For a kernel operator $K:S_{1} \to \mathcal{P}(S_{2})$ (that is a regular conditional probability from $S_1$ to $S_2$) for standard Borel spaces $S_1, S_2$, we define the Dobrushin coefficient as:
\begin{align*}
\delta(K)&=\inf\sum_{i=1}^{n}\min(K(x,A_{i}),K(y,A_{i}))\label{Dob_def}
\end{align*}
where the infimum is over all $x,y \in S_{1}$ and all partitions $\{A_{i}\}_{i=1}^{n}$ of $S_{2}$.
\end{definition}
%We note that this definition holds for continuous or finite/countable spaces $S_{1}$ and $S_{2}$ and $0\leq \delta(K)\leq 1$ for any kernel operator. 

Note that for the channel $O$ Dobrushin coefficient is readily defined, for the transition kernel $T$, let
\[ \tilde{\delta}({ T}):=\inf_{u \in \mathds{U}} \delta({T}(\cdot|\cdot,u)). \]

We now state the following corollary to Theorem \ref{main_result} and \cite[Theorem 3.3]{mcdonald2020exponential}.
\begin{corollary}\label{cor_filter}
Suppose Assumption \ref{main_assmp} holds and let  $\alpha:=(1-\tilde{\delta}(\mathcal{T}))(2-\delta(O))<1$. Let $X_0\sim\mu$ and consider $H_0^N=\{Y_0,\dots,Y_N,U_0,\dots,U_{N-1}\}$  with a policy $\hat{\gamma}$ acting on the first $N$ steps. For the policy found with the algorithm presented in Section \ref{app_alg}, $\hat{\gamma}^N$, acting after the first $N$ steps we can write that,
\begin{align*}
&E_{\mu}^{\hat{\gamma}}\left[J_\beta(\hat{Z}_0,\hat{\gamma}^N) -J^*_\beta(\hat{Z}_0)\right]\leq \frac{\|c\|_\infty}{(1-\beta)^2} \left(\beta\alpha_{\mathds{Y}} L_\mathds{Y}+ 4\alpha^N\right).
\end{align*}
\end{corollary}
\begin{proof}
\cite[Theorem 3.3]{mcdonald2020exponential} implies that 
\begin{align*}
L&_t:=\sup_{\hat{\gamma}\in\hat{\Gamma}}E_{\mu}^{\hat{\gamma},\hat{O}}\bigg[\|P^{\pi_t^-}(X_{t+N}\in\cdot|\hat{Y}_{[t,t+N]},U_{[t,t+N-1]})\nonumber\\
&\qquad\qquad\qquad\qquad-P^{\pi^*}(X_{t+N}\in\cdot|\hat{Y}_{[t,t+N]},U_{[t,t+N-1]})\|_{TV}\bigg]\\
&\leq 2 \left((1-\tilde{\delta}(\mathcal{T}))(2-\delta(\hat{O}))\right)^N.
\end{align*}
Hence, we only need to show that 
\begin{align*}
\delta(\hat{O})\geq \delta(O).
\end{align*}
However, this is a direct consequence of the definition of Dobrushin coefficient, since for any discrete channel $\hat{O}$, the quantization bins $\{B_i\}$ forms a partition for $\mathds{Y}$.
\end{proof}
Corollary \ref{cor_filter} separates the effects of discretization and finite memory use, in two different terms. However, we caution the reader that this is only done for an upper bound. In fact, we do not bound the filter stability term of the discrete channel $\hat{O}$, using the filter stability term of the original channel $O$. Instead, we bound an upper bound of the filter stability term for $\hat{O}$, using an upper bound for the filter stability term of the channel $O$. This can be done since \cite[Theorem 3.3]{mcdonald2020exponential} favors ``non-informative'' channels.

\subsubsection{Utilizing Finite Measurements via a Hilbert Projective Metric Based Bound for Filter Stability}\label{hil_sec}

Via a somewhat different, and more direct, derivation,\cite[Section 4.2 and Theorem 17]{kara2020near} presented the following alternative condition involving sample path-wise uniform filter stability term
 \begin{align}\label{TVUnifB}
&\bar{L}^N_{TV}:=\sup_{pi\in \mathcal{P}(\mathds{X})}\sup_{y_{[0,N]},u_{[0,N-1]}}  \left\|P^{\pi}(\cdot|y_{[0,N]},u_{[0,N-1]})-P^{\pi^*}(\cdot|y_{[0,N]},u_{[0,N-1]})\right\|_{TV},
\end{align}
where $\pi^*$ has been fixed in Section \ref{app_section}. 
%\yed{(The term $\sup_{\gamma\in\Gamma}$ has been removed as $\sup_{u_{[0,N-1]}}$ alone is sufficient.)} 
to show the following {\it uniform} error bound:
\begin{align}\label{jmlrboundF}
\sup_z\left|J_\beta(z,\gamma_N)-J^*_\beta(z)\right|\leq \frac{2(1+(\alpha_{\mathcal{Z}}-1)\beta)}{(1-\beta)^3(1-\alpha_{\mathcal{Z}}\beta)}\|c\|_\infty \bar{L}^N_{TV}
\end{align}
for all $\beta \in (0,1)$ under a contraction condition, for some constant $\alpha_{\mathcal{Z}}$ defined in \cite{kara2020near}.

Accordingly, a further method, and one which leads to complementary conditions given the above, for filter stability is via the Hilbert projective metric \cite[ Lemma 3.8]{le2004stability}. Via the Birkhoff-Hopf theorem, a controlled version of a contraction via the Hilbert metric \cite{le2004stability} can be utilized \cite{demirciRefined2023}:
% $\sup \frac{h(\mu K, \nu K)}{h(\mu,\nu)} =: K$ with $K = \frac{\sqrt{M}-1}{\sqrt{M}+1} < 1$, where $M = \sup h(\mu K, \nu K)$. 
%Thus, the Hilbert metric bound can be used as a further sufficient condition. $\epsilon \leq {\cal T}(x,y) \leq \frac{1}{\epsilons} < \infty$ implies contraction, we also need that $\epsilon \leq Q(y|x) \leq \frac{1}{\epsilons}$. CITE THE NEW PAPER BY EMRE AND ALI.

\begin{definition}[Mixing kernel]\label{mixingKernelDef}
    The non-negative kernel $K$ defined on $\mathbb{X}$ is mixing, if there exists a constant $0<\varepsilon \leq 1$, and a non-negative measure $\lambda$ on $\mathbb{X}$, such that
    $$
    \varepsilon \lambda(A) \leq K(x, A) \leq \frac{1}{\varepsilon} \lambda(A)
    $$
    for any $x \in \mathbb{X}$, and any Borel subset $A \subset \mathbb{X}$.
    \end{definition}
With
\[F(z, y, u)(\cdot)=\operatorname{Pr}\left\{X_{k+1} \in \cdot \mid Z_k=z, Y_{k+1}=y, U_k=u\right\},\]
we note the following:
\begin{assumption}\label{mixing_kernel_con}
    \begin{enumerate}
    \item $O(y|x)\geq \epsilon > 0$ for every $x\in \mathbb{X}$ and $y\in \mathbb{Y}$.
    \item The transition kernel $\mathcal{T}(.|.,u)$ is a mixing kernel (see Definition \ref{mixingKernelDef}) for every $u\in\mathbb{U}$.
\end{enumerate} 
\end{assumption}

\begin{lemma}\label{clm}\cite{demirci2023geometric}
    Under Assumption \ref{mixing_kernel_con}, 
    there exists a constant $r < 1$ such that
    % \begin{align}
    %     h(F(\mu, y,u), F(\nu, y,u))\leq \frac{1-\epsilon_{u}^2 }{1+\epsilon_{u}^2 } h(\mu, \nu)
    % \end{align}
    \begin{align}
        h(F(\mu, y,u), F(\nu, y,u))\leq r h(\mu, \nu)
    \end{align}
   for every comparable $\mu,\nu\in {\cal P}(\mathbb{X})$ and for every
    $u\in \mathbb{U}$ and $y\in \mathbb{Y}$   where $h$ is the Hilbert projective metric. Here $r=\frac{1-\epsilon_{u}^2\epsilon }{1+\epsilon_{u}^2\epsilon},$ $\epsilon_{u}$ is the mixing constant of the kernel $\mathcal{T}(.|.,u)$.
    % , where
    % $\epsilon_{u}$
    % is the mixing constant of the kernel $\T(.|.,u)$.
\end{lemma}

\begin{theorem}\label{d}\cite{demirciRefined2023}
    Under Assumption \ref{mixing_kernel_con}, 
    there exists a constant $r<1$ and $K$ such that 
    \begin{align}
        \bar{L}_{TV}^N \leq r^{N-1} K.
    \end{align}
    Here, $K=\frac{2}{\log 3} \sup h(\pi_1, \pi_1^*)$ and $r=\sup_{u\in \mathbb{U}} \frac{1-\epsilon_{u}^2\epsilon }{1+\epsilon_{u}^2\epsilon}$ where $\pi_1 = F(\pi_0,y,u)$ and $\pi^*_1= F(\pi^*,y,u)$ and the supremum is taken over all $\pi_0,\pi^*,y,u$. 
\end{theorem}

\begin{corollary}\label{cor_kara_y}\cite{demirciRefined2023}
 Under Assumption \ref{mixing_kernel_con}, 
    there exists a constant $r<1$ and $K$ such that 
    \begin{align}
    E_{\pi_0^{-}}^{\hat{\gamma}}\left[\left|\tilde{J}_\beta^N\left(\hat{Z}_0, \tilde{\phi}^N\right)
    -J_\beta^*\left(\hat{Z}_0\right)\right| \right] 
    \leq \frac{2 \|c\|_{\infty}}{(1-\beta)^2} r^{N-1} K.
    \end{align}
        Here, $K=\frac{2}{\log 3} \sup_{\pi_0^-,y,u} h(F(\pi_0^-.y,u),F(\pi^*,y,u))$ and $r=\sup_{u\in \mathbb{U}} \frac{1-\epsilon_{u}^2\epsilon }{1+\epsilon_{u}^2\epsilon}$.
\end{corollary}

The implication of the quantized approximation is that we can ensure that upon quantization each bin can satisfy the condition that $\hat{O}(y|x) > \epsilon > 0$. That is, the original measurement kernel may not satisfy the contraction condition but the finite measurement model may, leading to the applicability of the bound above. Note that the conditions are complementary when compared with those building on the Dobrushin coefficient based analysis above.

\subsection{Comparison of Discretization Rate and Finite Window Length}

Consider, for example, the original channel $O$, is fully ``non-informative'' and provides a random observation $y$ uniformly distributed on $\mathds{Y}$, then the discretization of the observation space will not lead to any improvement as the Lipschitz constant, $\alpha_{\mathds{Y}}=0$. However, using a longer window, will improve the performance. In particular, for this case the Dobrushin coefficient of the channel becomes $\delta(O)=1$, thus $\alpha=(1-\delta(\tilde{T}))$.

For the other extreme case, we consider a ``fully informative'' channel. Since, we have a density assumption on the channel, consider instead an additive channel of the from $y_t=x_t+v_t$ where $v_t$ is a zero mean Gaussian with small variance. Note that, although we assume $\mathds{Y}$ to be compact, we can extend the analysis by focusing on high probability events on compact subsets. For this case, finer quantization of the observation space will lead to greater improvements on the performance as the Lipschitz constant of the density of the channel will be large. However, as the channel is already informative, using a longer memory will not provide further information on the hidden state $x$. 

For the extreme cases, we observe that for ``fully informative'' channels it is better to increase the discretization rate rather than increasing the window length. For ``fully non-informative'' channels, however, one should increase the window length instead of increasing the discretization rate. Even though, we observe this trend for the extreme cases, for general observation models, one needs give a quantitative definition of an ``informative'' channel. \cite{Ramon2008discrete,mcdonald2018stability} study stochastic observability for measurement channels for non-linear dynamics and establish the filter stability. This notion can be used to further compare the effect of discretization and memory use for ``observable'' channels. We should also note that increasing the size of the finite observation set $\hat{\mathds{Y}}_M$ from $M$ to $M+1$, will increase the size of the state space of the fully observed MDP constructed by the algorithm in Section \ref{app_alg} by a factor of $\left(\frac{M+1}{M}\right)^N$, and increasing the window size from $N$ to $N+1$ will increase the state space size by a factor of $M\times|\mathds{U}|$.
\begin{remark} We note that if the goal is only asymptotic convergence of the approximation error, we can relax the geometric convergence conditions of the filter stability (and only consider stochastic observability of the controlled kernels (\cite{mcdonald2018stability}) so that $L_t\to 0$ as the window length increases with no rate of convergence),  $\mathds{U}$ is compact (not necessarily finite) and the density of the measurement kernels, denoted with $f$ is only continuous in $y$, with no Lipschitz regularity. \end{remark}

\section{Q-Learning for the Approximate Models}\label{QLearningApproximateM}

\subsection{Q-Learning using Belief Discretization}
We will first introduce a discretization based Q learning over the belief spaces. We first set a finite number of quantization bins, $\{Z_1,\dots,Z_k\}\subset\P(\mathds{X})$ such that $\cup_{i=1}^k = \P(\mathds{X})$. We use the map $\Psi:\P(\mathds{X}) \to \{Z_1,\dots,Z_k\}$ for quantization. 

A natural way of choosing the quantization for the belief space is as follows: (i) quantize the hidden space $\mathds{X}$ to a finite set, (ii) for the resulting simplex use further discretization over the probability mass of the resulting bins (see \cite{saldi2019pomdp}).

Finally, we use the following Q learning algorithm over the finite set of quantization bins. We denote the $i$ to denote the quantization bin $Z_i$:
\begin{align*}
&Q_{k+1}(i,u)=(1-\alpha_k(i,u))Q_k(i,u)+\alpha_k(i,u)\left(C_k+\beta \min_v Q_k(\psi(\pi_{t+1}),v)\right).
\end{align*}
This algorithm can be shown to converge under suitable ergodicity conditions on the belief process (see \cite[Section 3.4]{karayukselNonMarkovian}). Furthermore, the near optimality of the learned policies can be guaranteed using the weak continuity of the belief kernel (\cite[Section 3.3]{KSYContQLearning}). 

We emphasize, however, that this algorithm assumes that the controller has access to the belief process $\{\pi_t\}$. In other words, the controller can map the the history $h_t$ to the belief state $\pi_t$, and this mapping can be challenging in general.

\subsection{Q learning Using Finite Memory with Discrete Observations}

In this section, we show that the proposed approximate model using the finite memory can be learned using a Q learning algorithm.

Assume that we start keeping track of the last $N+1$ observations and the last $N$ control action variables after at least $N+1$ time steps. That is, at time $t$, we keep track of the information variables
\begin{align*}
I^N_t=\begin{cases}&\{y_t, y_{t-1},\dots,y_{t-N},u_{t-1},\dots,u_{t-N}\} \quad \text{ if } N>0\\
&y_t  \quad \text{ if } N=0.\end{cases}
\end{align*}
For every observation $y_t$, we use the representative element from the finite $\hat{\mathds{Y}}_M$, using the map $\psi_{\mathds{Y}}$ (see (\ref{quant_map})). With an abuse of notation, we use $\psi_{\mathds{Y}}(h_t^N)$ to denote the finite history information with discrete observations. 

For these new approximate states, we follow the usual Q learning algorithm such that for any $I\in\mathds{Y}^{N+1}\times\mathds{U}^N$ and $u\in\mathds{U}$
\begin{align}\label{q_alg}
&Q_{k+1}(\psi_{\mathds{Y}}(I),u)=(1-\alpha_k(\psi_{\mathds{Y}}(I),u))Q_k(\psi_{\mathds{Y}}(I),u)\nonumber\\
&+\alpha_k(\psi_{\mathds{Y}}(I),u)\left(C_k+\beta \min_v Q_k(\psi_{\mathds{Y}}(I_1),v)\right),
\end{align}
where $I_1 = \{Y_{t+1}, y_{t},\dots,y_{t-N+1},u_{t},\dots,u_{t-N+1}\}$ which are the finite memory information variables we observe following $I$ and the control $u$. $C_k$ denotes the cost we observe at time time $k$ after we apply $u$, hence, if the hidden state is $x_k$, we have $C_k=c(x_k,u)$.
%We note that the update times are different for each $(I,u)$ pair, that is, $Q_k(I,u)$ is updated only when the process hits $(I,u)$.

To choose the control actions, we use polices that choose the control actions randomly and independent of everything else such that at time $t$ $u_t= u_i, \text{ w.p } \sigma_i$ for any $u_i\in\mathds{U}$ with $\sigma_i>0$ for all $i$.

\begin{assumption}\label{partial_q}
\hfill
\begin{itemize}
\item [1.] $\alpha_t(\psi_{\mathds{Y}}(I),u)=0$ unless $(\psi_{\mathds{Y}}(h_t),u_t)=(\psi_{\mathds{Y}}(I),u)$. Furthermore,
\[\alpha_t(\psi_{\mathds{Y}}(I),u) = {1 \over 1+ \sum_{k=0}^{t} 1_{\{\psi_{\mathds{Y}}(I_k)=\psi_{\mathds{Y}}(I), u_k=u\}} }\]
We note that, this means that $\alpha_k(\psi_{\mathds{Y}}(I),u)=\frac{1}{k}$ if $\psi_{\mathds{Y}}(I_k)=\psi_{\mathds{Y}}(I),u_k=u$, if $k$ is the instant of the $k$th visit to $(\psi_{\mathds{Y}}(I),u)$.% as this will be crucial in the averaging of the Markov chain dynamics (see Remark \ref{onLearningRates}).
\item [2.] Under every stationary \{memoryless or finite memory exploration\} policy, say $\gamma$, the true state process, $\{X_t\}_t$, is positive Harris recurrent and in particular admits a unique invariant measure $\pi_\gamma^*$.
\item [3.] During the exploration phase, every $(\psi_{\mathds{Y}}(I),u)$ pair is visited infinitely often.
\end{itemize}
\end{assumption}

%\begin{assumption}\label{partial_q}
%\begin{itemize}
%\item[1.] $\alpha_t(z,u)=\frac{1}{t}$.

%\item [2.] Under every stationary policy, the true state process, $\{x_t\}_t$, admits an invariant measure $\pi_\gamma^*$.

%\item[3.] The transition kernel $\mathcal{T}(dx_1|x_0,u_0)$ admits a density function $f(x_1,x_0,u_0)$ such that $\mathcal{T}(dx_1|x_0,u_0)=f(x_1,x_0,u_0)\phi(dx_0)$ for some reference measure $\phi$ and $f(x_1,x_0,u_0)$ is Lipschitz continuous in $x_0$ with some Lipschitz constant $k<\infty$.
%\end{itemize}
%\end{assumption}

\begin{theorem}\label{main_thm}
Under Assumption \ref{partial_q},
\begin{itemize}
\item[i.] The algorithm given in (\ref{q_alg}) converges almost surely to some $Q^*$ which are the optimal Q values of an MDP constructed in Section \ref{app_section} with the discrete channel $\hat{O}$.
\item[ii.] For any policy $\gamma^N$ that is constructed using $Q^*$ if we assume that the controller has access to at least $N+1$ observations and $N$ control action variables, when it starts acting, we have
\begin{align*}
&J_\beta(\mu,{O},\gamma^N)-J_\beta^*(\mu,{O})\\
&\leq \frac{\beta}{(1-\beta)^2}\|c\|_\infty \alpha_{\mathds{Y}} L_\mathds{Y}+ \frac{2\|c\|_\infty}{(1-\beta)}\sum_{t=0}^\infty \beta^t L_t,
\end{align*}
\end{itemize}
\end{theorem}

\begin{proof}
The result is a direct implication of \cite[Theorem 4.1]{kara2021convergence}. In \cite[Theorem 4.1]{kara2021convergence}, it is shown that the Q learning using finite memory information variables for a POMDP with a finite observation space, converges. The Q iterations constructed here then can be thought to learn a POMDP model with the discrete channel $\hat{O}$ which has a finite observation space. Hence, (ii) is a result of Theorem \ref{main_result}.
\end{proof}

\sy{The positive Harris recurrence condition can be relaxed; please see \cite{karayukselNonMarkovian} \cite[Lemma 6]{creggZeroDelayNoiseless}}.

\section{Concluding Remarks and a Discussion}

We studied two approximation algorithms for POMDPs with continuous state and observation spaces. The first method discretizes the hidden state space and genrates a belief MDP defined over a simplex. The near optimality of this method relies on the regularity of the transition and the channel kernel of the POMDP. Furthermore, this method is tailored more towards settings where the model is known. The second method is done by dicretizing the observation space and using finite memory of discrete information variables. The near optimality of this method relies on the controlled filter stability under the discretized observations. Furthermore, we provided a Q learning algorithm for the constructed  approximate finite models.

\appendix 

\section{Proof of Lemma \ref{imm_lem}}\label{imm_lem_proof}
Consider some $x\in A_i$, we then write
\begin{align*}
W_1\left(\mathcal{T}(\cdot|x,u), \hat{\mathcal{T}}(\cdot|x,u)\right)& = \sup_{\|f\|_{Lip}\leq 1} \left|\int f(x_1)\mathcal{T}(dx_1|x,u) - \int f(x_1)\hat{\mathcal{T}}(dx_1|x,u)\right|\\
&=\sup_{\|f\|_{Lip}\leq 1} \left|\int f(x_1)\mathcal{T}(dx_1|x,u) - \int f(x_1){\mathcal{T}}(dx_1|z,u)\pi_i(dz)\right|\\
&\leq\sup_{\|f\|_{Lip}\leq 1} \int_{z\in A_i} \left| \int_{x_1\in \mathds{X}} f(x_1)\mathcal{T}(dx_1|x,u) - \int f(x_1){\mathcal{T}}(dx_1|z,u)\right|\pi_i(dz)\\
&\leq \int_{z\in A_i}W_1\left(\mathcal{T}(\cdot|x,u), {\mathcal{T}}(\cdot|z,u)\right)\pi_i(dz)\\
& \leq K_T \sup_{z\in A_i}\|x-z\| \leq K_T L_{\mathds{X}}.
\end{align*}
The proof for the bound on the approximate channel follows from identical steps.

\section{Proof of Lemma \ref{iter_lem}}\label{iter_lem_proof}
We start by defining the following functions for rotational convenience. 
\begin{align*}
F_k(x_{t-k},y_{[0,t-k]}) = E\left[ c(X_t,\gamma(Y_{[0,t]})) |  x_{t-k},y_{[0,t-k]}\right]\\
\hat{F}_k(x_{t-k},y_{[0,t-k]}) = \hat{E}\left[ \hat{c}(X_t,\gamma(Y_{[0,t]})) |  x_{t-k},y_{[0,t-k]}\right]
\end{align*}
in particular
\begin{align*}
&F_0(x_t,y_{[0,t]}) = c(x_t,y_{[0,t]}), \qquad \hat{F}_0(x_t,y_{[0,t]}) = \hat{c}(x_t,y_{[0,t]}).
\end{align*}
Furthermore, we are interested in the difference $\left|E[F_t(x_0,y_0)]-\hat{E}[\hat{F}_t(x_0,y_0)]\right|$.

\adk{
We first note the following due to iterated expectations 
\begin{align}\label{it_exp}
&F_k(x_{t-k},y_{[0,t-k]}) = E\left[ c(X_t,\gamma(Y_{[0,t]})) |  x_{t-k},y_{[0,t-k]}\right] \nonumber\\
&= E\left[ E\left[ c(X_t,\gamma(Y_{[0,t]}))| X_{t-k+1}, Y_{[0,t-k+1]}  \right] | x_{t-k},y_{[0,t-k]}\right]\nonumber\\
&=E\left[ F_{k-1}(X_{t-k+1}, Y_{[0,t-k+1]})  | x_{t-k},y_{[0,t-k]} \right]\nonumber\\
&= \int F_{k-1}(x_{t-k+1}, y_{[0,t-k+1]})  O(dy_{t-k+1}|x_{t-k+1}) \mathcal{T}(dx_{t-k+1}|x_{t-k},\gamma(y_{[0,t-k]}))
\end{align}}
which can also be written for the functions $\hat{F}_k$. Using this iterative form, we  find an upper-bound for the Lipschitz constant of $F_k(x_{t-k},y_{[0,t-k]}) $ with respect to $x_{t-k}$ uniformly over all $y_{[0,t-k]}$:
\begin{align*}
&\left|F_k(x_{t-k},y_{[0,t-k]})  - F_k(x'_{t-k},y_{[0,t-k]}) \right| \\
&= \bigg|\int F_{k-1}(x_{t-k+1}, y_{[0,t-k+1]})  O(dy_{t-k+1}|x_{t-k+1}) \mathcal{T}(dx_{t-k+1}|x_{t-k},\gamma(y_{[0,t-k]}))\\
&\qquad - \int F_{k-1}(x_{t-k+1}, y_{[0,t-k+1]})  O(dy_{t-k+1}|x_{t-k+1}) \mathcal{T}(dx_{t-k+1}|x'_{t-k},\gamma(y_{[0,t-k]}))\bigg|\\
&\leq  \|\int F_{k-1}(x_{t-k+1}, y_{[0,t-k+1]}) O(dy_{t-k+1}|x_{t-k+1})\|_{Lip} K_T |x_{t-k} - x'_{t-k}|\\
&\leq \left(K_O \|c\|_\infty+ \|F_{k-1}\|_{Lip}  \right) K_T \|x_{t-k} - x'_{t-k}\|
\end{align*}
where we used Assumption \ref{channel_kernel_reg}(b) for the first inequality. For the last step, we used the following:
\begin{align*}
&\bigg|\int F_{k-1}(x_{t-k+1}, y_{[0,t-k+1]}) O(dy_{t-k+1}|x_{t-k+1})  \int F_{k-1}(x'_{t-k+1}, y_{[0,t-k+1]}) O(dy_{t-k+1}|x'_{t-k+1})\bigg|\\
&\leq \bigg|\int F_{k-1}(x_{t-k+1}, y_{[0,t-k+1]}) O(dy_{t-k+1}|x_{t-k+1}) - \int F_{k-1}(x_{t-k+1}, y_{[0,t-k+1]}) O(dy_{t-k+1}|x'_{t-k+1})\bigg|\\
&  + \bigg|\int F_{k-1}(x_{t-k+1}, y_{[0,t-k+1]}) O(dy_{t-k+1}|x'_{t-k+1}) - \int F_{k-1}(x'_{t-k+1}, y_{[0,t-k+1]}) O(dy_{t-k+1}|x'_{t-k+1})\bigg|\\
&\leq \|F_{k-1}\|_\infty \|O(\cdot|x_{t-k+1}) - O(\cdot|x'_{t-k+1})\|_{TV} + \|F_{k-1}\|_{Lip}  \|x_{t-k+1} - x'_{t-k+1}\|\\
& \leq\left( K_O \|c\|_\infty+  \|F_{k-1}\|_{Lip} \right)\|x_{t-k+1} - x'_{t-k+1}\|.
\end{align*}
One can then show that 
\begin{align}\label{lip_bound}
\|F_k\|_{Lip}&\leq K_O\|c\|_\infty\left( K_T +\dots  + K_T^k\right) + \|F_0\|_{Lip}K_T^k\nonumber\\
&=  K_O\|c\|_\infty\left( K_T +\dots  + K_T^k\right) + K_c K_T^k
\end{align}
%We now focus on the functions $F_k$ and $\hat{F}_k$ to write them in an iterative way:
%\begin{align*}
%&F_k(x_{t-k},y_{[0,t-k]}) = \int F_{k-1}(x_{t-k+1}, y_{[0,t-k+1]}) O(dy_{t-k+1}|x_{t-k+1})\mathcal{T}(dx_{t-k+1}|x_{t-k},\gamma(y_{[0,t-k]}))\\
%&\hat{F}_k(x_{t-k},y_{[0,t-k]}) = \int \hat{F}_{k-1}(x_{t-k+1}, y_{[0,t-k+1]}) \hat{O}(dy_{t-k+1}|x_{t-k+1})\hat{\mathcal{T}}(dx_{t-k+1}|x_{t-k},\gamma(y_{[0,t-k]}))
%\end{align*}
The iterative form in (\ref{it_exp}) implies that
\begin{align*}
&\left|F_k(x_{t-k},y_{[0,t-k]}) - \hat{F}_k(x_{t-k},y_{[0,t-k]})\right|\\
&= \bigg|  \int F_{k-1}(x_{t-k+1}, y_{[0,t-k+1]}) O(dy_{t-k+1}|x_{t-k+1})\mathcal{T}(dx_{t-k+1}|x_{t-k},\gamma(y_{[0,t-k]}))\\
&\qquad -  \int \hat{F}_{k-1}(x_{t-k+1}, y_{[0,t-k+1]}) \hat{O}(dy_{t-k+1}|x_{t-k+1})\hat{\mathcal{T}}(dx_{t-k+1}|x_{t-k},\gamma(y_{[0,t-k]}))\bigg|\\
&\leq \bigg|  \int F_{k-1}(x_{t-k+1}, y_{[0,t-k+1]}) O(dy_{t-k+1}|x_{t-k+1})\mathcal{T}(dx_{t-k+1}|x_{t-k},\gamma(y_{[0,t-k]}))\\
&\qquad -  \int {F}_{k-1}(x_{t-k+1}, y_{[0,t-k+1]}) {O}(dy_{t-k+1}|x_{t-k+1})\hat{\mathcal{T}}(dx_{t-k+1}|x_{t-k},\gamma(y_{[0,t-k]}))\bigg|\\
& + \bigg|  \int F_{k-1}(x_{t-k+1}, y_{[0,t-k+1]}) O(dy_{t-k+1}|x_{t-k+1})\hat{\mathcal{T}}(dx_{t-k+1}|x_{t-k},\gamma(y_{[0,t-k]}))\\
&\quad\qquad -  \int \hat{F}_{k-1}(x_{t-k+1}, y_{[0,t-k+1]}) \hat{O}(dy_{t-k+1}|x_{t-k+1})\hat{\mathcal{T}}(dx_{t-k+1}|x_{t-k},\gamma(y_{[0,t-k]}))\bigg|\\
&\leq   \|\int F_{k-1}(x_{t-k+1}, y_{[0,t-k+1]}) O(dy_{t-k+1}|x_{t-k+1})\|_{Lip} \sup_{x,u}W_1(\mathcal{T}(\cdot|x,u),\hat{\mathcal{T}}(\cdot|x,u))\\
&\qquad + \sup_x \left| \int F_{k-1}(x, y_{[0,t-k+1]}) O(dy_{t-k+1}|x) - \int \hat{F}_{k-1}(x, y_{[0,t-k+1]}) \hat{O}(dy_{t-k+1}|x)\right|\\
& \leq \sup_{x,u}W_1(\mathcal{T}(\cdot|x,u),\hat{\mathcal{T}}(\cdot|x,u))\left( \|F_{k-1}\|_{Lip} + K_O \|c\|_\infty\right) + \|c\|_\infty\sup_x\|O(\cdot|x) - \hat{O}(\cdot|x)\|_{TV} + \|F_{k-1} - \hat{F}_{k-1}\|_\infty\\
&\leq K_T L_{\mathds{X}}  \|F_{k-1}\|_{Lip} + K_T K_O\|c\|_\infty L_{\mathds{X}} + K_O\|c\|_\infty L_{\mathds{X}} + \|F_{k-1} - \hat{F}_{k-1}\|_\infty
\end{align*}
where we used Lemma \ref{imm_lem}. Using (\ref{lip_bound}), we can find a further upper-bound as:
\begin{align*}
&\left\|F_k - \hat{F}_k\right\|_\infty\leq K_T L_{\mathds{X}} \left( K_O\|c\|_\infty\left( K_T +\dots  + K_T^{k-1}\right) + K_c K_T^{k-1}\right) \\
&\qquad\qquad\qquad+  K_T K_O\|c\|_\infty L_{\mathds{X}} + K_O\|c\|_\infty L_{\mathds{X}} + \|F_{k-1} - \hat{F}_{k-1}\|_\infty\\
&\leq K_O\|c\|_\infty L_\mathds{X} \left( K^2_T +\dots  + K_T^{k}  \right) + K_cK_T^kL_\mathds{X} +K_O\|c\|_\infty K_TL_\mathds{X} + K_O\|c\|_\infty L_\mathds{X} +  \|F_{k-1} - \hat{F}_{k-1}\|_\infty \\
&= K_O \|c\|_\infty L_\mathds{X}\sum_{m=0}^k (K_T)^m  + K_c (K_T)^{k} L_\mathds{X} + \| F_{k-1} - \hat{F}_{k-1}\|_\infty.
\end{align*}
Denoting by $A_k:=\sum_{m=0}^k K_T^m$, we can then write:
\begin{align*}
\left\|F_k - \hat{F}_k\right\|_\infty&\leq K_O\|c\|_\infty L_\mathds{X} \sum_{n=1}^k A_n + K_c L_\mathds{X}\sum_{n=1}^k (K_T)^n  + \|F_0 - \hat{F}_0\|_\infty\\
&\leq K_O \|c\|_\infty L_\mathds{X} \sum_{n=1}^k A_n + K_c L_\mathds{X} \sum_{n=1}^k(K_T)^n  + K_c L_\mathds{X}\\
&\leq  K_O\|c\|_\infty L_\mathds{X} \sum_{n=0}^k A_n + K_c L_\mathds{X}\sum_{n=0}^k (K_T)^n \\
&= K_O\|c\|_\infty L_\mathds{X} \sum_{n=0}^k \sum_{m=0}^n K_T^m + K_c L_\mathds{X}\sum_{n=0}^k K_T^n.
\end{align*}

\bibliographystyle{plain}

\bibliography{SerdarBibliography,AliBibliography,references_acc,references}

\end{document}